\documentclass{article}

\usepackage[labelfont=bf]{caption}

\usepackage{PRIMEarxiv}

\usepackage{graphicx}

\usepackage[utf8]{inputenc} 
\usepackage[T1]{fontenc}    
\usepackage{amsmath,amsthm}
\usepackage{amssymb}
\usepackage{hyperref}       
\usepackage{url}            
\usepackage{booktabs}       
\usepackage{amsfonts}       
\usepackage{nicefrac}       
\usepackage{microtype}      
\usepackage{cleveref}       
\usepackage{lipsum}         
\usepackage{graphicx}
\usepackage[round]{natbib}
\setcitestyle{authoryear,open={(},close={)}}
\usepackage{doi}

\usepackage{comment}
\usepackage{bm}
\usepackage{multirow}
\usepackage{setspace}
\usepackage{xcolor}
\usepackage{caption}
\usepackage{subcaption}
\usepackage{algpseudocode}

\usepackage[ruled,linesnumbered]{algorithm2e}

\newtheorem{lemma}{Lemma}
\newtheorem{conjecture}{Conjecture}

\newtheorem{theorem}{Theorem}
\newtheorem{proposition}{Proposition}
\newtheorem{example}{Example}

\title{Picking Operations in Warehouses with
Dynamically Arriving Orders: How Good is
Reoptimization? }

\author{
  Catherine Lorenz\\
  Chair of Management Science/Operations and Supply Chain Management,\\ University of Passau, Passau, Germany\\
   \And
  Alena Otto\\
  Chair of Management Science/Operations and Supply Chain Management,\\ University of Passau, Passau, Germany\\
  \texttt{Alena.Otto@uni-passau.de} \\
   \And
  Michel Gendreau\\
  Department of Mathematics and Industrial Engineering and CIRRELT,\\
Polytechnique Montr\'{e}al, Montr\'{e}al, Canada
}

\begin{document}
\maketitle

\begin{abstract}
E-commerce operations are essentially online, with customer orders arriving dynamically. However, very little is known about the performance of online policies for warehousing with respect to optimality, particularly for order picking and batching operations, which constitute a substantial portion of the total operating costs in warehouses.
We aim to close this gap for one of the most prominent dynamic algorithms, namely  
 reoptimization \textit{(Reopt)}, which reoptimizes the current solution each time when a new order arrives. We examine \textit{Reopt} in \textit{the Online Order Batching, Sequencing, and Routing Problem (OOBSRP)}, in both cases when the picker uses either a manual pushcart or a robotic cart. Moreover, we examine the \textit{noninterventionist Reopt} in the case of a manual pushcart, wherein picking instructions are provided exclusively at the depot.
We establish \textit{analytical} performance bounds employing \textit{worst-case and probabilistic analysis}. We demonstrate that, under generic stochastic assumptions, \textit{Reopt} is almost surely asymptotically optimal and, notably, we validate its near-optimal performance in computational experiments across a broad range of warehouse settings. These results underscore \textit{Reopt}'s relevance as a method for online warehousing applications.
\end{abstract}


\section{Introduction}
E-commerce has emerged as an important channel of retail trade, emphasizing fast deliveries and large product assortments, while pressuring for low logistics costs. This makes warehousing an operational challenge and a leading determinant of competitiveness in the e-commerce market. 
The central component of efficient warehousing is order picking, which accounts for up to 60-70\% of the total operating costs \citep{Chen2015,Marchet2015}. 
Most warehouses are so-called \textit{picker-to-parts} warehouses \citep{napolitano2012, Vanheusden2022}, in which order pickers travel around the warehouse to collect inventory items (see Figure~\ref{fig:carts}a).  In a widespread \textit{pick-while-sort} setup, pickers try to collect several orders simultaneously to save on costs. They use a \textit{manual pushcart} (Figure~\ref{fig:carts}a) or a \textit{robotic cart} (Figure~\ref{fig:carts}b)
.  Each cart contains several bins, and each bin can accommodate an order. In this paper, we focus on picking operations in picker-to-parts warehouses and examine both technologies, manual pushcarts and robotic carts.

\begin{figure}[h]

    \begin{subfigure}[b]{0.49\textwidth}

        \includegraphics[width=\textwidth]{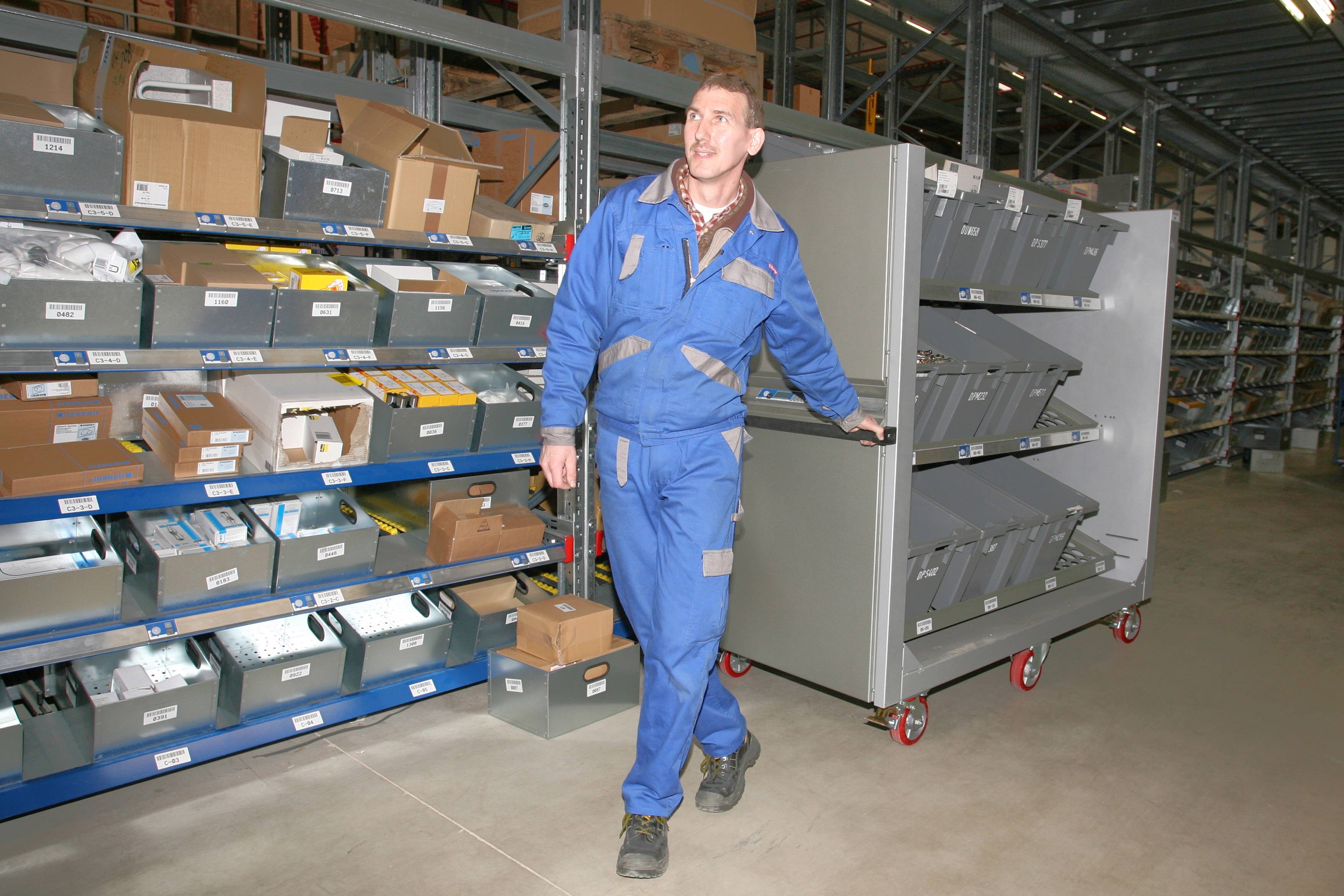}
        \raggedright\tiny{Source: KBS Industrieelektronik GmbH. License: CC BY-SA 3.0}
        \vspace{0.1cm}
        \caption*{(a) Pushcart}
        \label{fig:pushcart}
    \end{subfigure}
    \hfill
    \begin{subfigure}[b]{0.49\textwidth}

        \includegraphics[width=\textwidth]{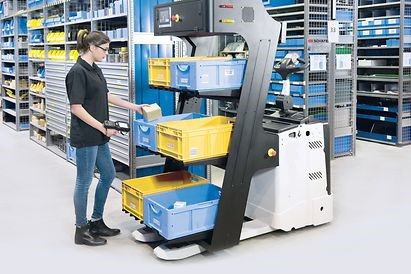}
        \raggedright\tiny{Source: SSI Schäfer}
        \vspace{0.1cm}
        \caption*{(b) Robotic cart}
        \label{fig:robot}
    \end{subfigure}

    \caption{An order picker with a cart in a picker-to-parts warehouse \label{fig:carts}}
    
\end{figure}

E-commerce operations are \textit{essentially online}, with customer orders arriving dynamically. However, \textit{very little} is known about the performance of the \textit{online policies} for handling these orders with respect to \textit{optimality}. A prominent online policy is \textit{reoptimization}, which myopically optimizes picking operations for the set of available orders at specific points in time. Moreover, reoptimization can leverage the recent progress in exact algorithms and heuristics developed for the 
\textit{offline} setting, in which all the customer orders are \textit{known} in advance \citep[cf. ][]{Loeffler2023, Schiffer2022, Valle2017}. Nevertheless, it remains open whether reoptimization makes sense for picking operations in warehouses. Optimization literature provides examples both of good \citep{Jaillet2008,Hwang2017} and unfortunate \citep[cf. ][]{Borodin1998} performance of reoptimization in online problems. 

In this paper, we study reoptimization analytically and experimentally in a generic order picking setting, which is known as \textit{the Online Order Batching, Sequencing, and Routing Problem (OOBSRP)} in the taxonomy of \citet{Pardo2023}. In OOBSRP, given the dynamically arriving orders, we have to determine \textit{a) [batching]} which orders to pick together in a single cart tour and \textit{b) [routing]} how to schedule the picking operations and how to route the picker. In the following, Section~\ref{sec:literature} reviews the literature, and we state the article's contribution in Section~\ref{sec:contribution}.

\subsection{Literature review}\label{sec:literature}
\textit{Online} order picking and batching problems have attracted much interest in the literature, see the surveys of \citet{Boysen2019, Li2022, Pardo2023, VanGils2019} and \cite{Vanheusden2022}. Yet, very little is known about \textit{optimal} policies. The available research can be roughly classified into three threads. The first one analyzes intuitive policies (such as `first come first served', S-shape routing, etc.). Here, important insights on the system stability and policy parameters are formulated using, in most cases,  the tools of the queuing theory \citep[cf. ][]{LeDuc2007,Vannieuwenhuyse2009,Yu2009,Yu2008,Schleyer2012}.  Within this thread, \citet{Bukchin2012}  analyze optimal decisions on the waiting time before picking a batch for a significantly simplified problem (e.g., the time to collect a batch does not depend on the locations of the respective items in the warehouse).  
Secondly, a number of papers design sophisticated online policies based on metaheuristics, AI, or rule-based approaches \citep[e.g., ][]{Cals2021, DHaen2022, Zhang2016,Zhang2017}.  Finally and prominently, a number of authors use reoptimization as the online policy of choice 
\citep[see][]{Lu2016, Giannikas2017, Dauod2022}. 
In all these studies, reoptimization is performed \textit{immediately} at the arrival of a new order.
In all the discussed cases, the designed policies are usually benchmarked against simple policies, so that 
their \textit{performance gap to optimality is generally unknown}. 

To the best of our knowledge, only \citet{Henn2012} and \citet{Alipour2018} compare selected policies to optimality. Both papers focus on the worst-case performance using \textit{competitive analysis}. 
\citet{Henn2012} studies reoptimization in a special case of OOBSRP with a manual pushcart, where the picker routes start and end in the depot and follow a predefined simplified scheme, which defines the problem as the \textit{Online Order Batching Problem (OOBP)} according to the classification of \citet{Pardo2023}. In his analysis, 
reoptimization can only be performed when the picker retrieves a new cart in the depot; it is performed either immediately when some recently arrived orders are available or after some deliberate waiting time. The concept of waiting times seems appealing, because further orders may arrive that make better-matching batches with the available orders. In \citet{Henn2012}, the introduction of the waiting time did not improve the results of the algorithm. The paper establishes a \textit{lower bound} of 3/2 for the \textit{strict competitive ratio} for any deterministic online algorithm and proves an \textit{upper bound} of 2 for the \textit{asymptotic competitive ratio} for their reoptimization policy (see definitions in Section~\ref{sec:comp_ratio}). The tight competitive ratios remain unknown. 
\citet{Alipour2018} extend the results of \citet{Henn2012} to multiple pickers. 

On a more general note, OOBSRP can be described as a nontrivial blend of two well-known optimization problems -- \textit{the online traveling salesman problem (online TSP)} (since the picker is routed over a set of locations) and \textit{the online batching problem (OBP)} (since the incoming orders have to be grouped into batches to be processed simultaneously in one cart tour). 
The seminal work of \citet{Ausiello1999} introduces two versions of online TSP: \textit{the nomadic one (N-TSP)}, for which the route of the online server ends at the last visited location, and \textit{the homing one (H-TSP)}, which requires the server to return to the depot. 
As we will explain in Section~\ref{sec:problem}, N-TSP and H-TSP resemble OOBSRP in the special case of infinite cart capacity in the cases where a robotic cart and a manual pushcart are used for picking, correspondingly. 
In OBP, the machine can process up to $c\in \mathbb{N}$ jobs simultaneously, and the processing time equals the processing time of the longest job. The objective is to schedule the processing of dynamically arriving non-preemptive jobs on this machine in order to minimize the makespan \citep[cf. ][]{Zhang2001}. OBP can be interpreted as a special case of OOBSRP with a manual pushcart in a single-aisle warehouse with the depot located at the end of the aisle. Let us denote as \textit{Reopt} the \textit{immediate} reoptimization policy, which reoptimizes each time a new order arrives. The \textit{worst-case} performance of \textit{Reopt} has been estimated as follows: The strict competitive ratio of \textit{Reopt} equals 2.5 for N-TSP \citep{Ausiello1999}, does not exceed 2.5 for H-TSP \citep{Ausiello1999}, and equals 2 for OBP \citep{Liu2000}. Moreover, although better policies are possible, \textit{Reopt}'s results in these problems are quite close to the best achievable worst-case results \citep[see][]{Ascheuer2000, Ausiello1999, Zhang2001}. Further, \citet{Ascheuer2000} proved that the result of \textit{Reopt} in H-TSP is at most twice the complete-information optimum objective value plus a constant additive factor that characterizes the diameter of the studied metric space. In none of the aforementioned studies, \textit{Reopt} was compared empirically to an optimal algorithm. 

Only very few papers in the literature perform probabilistic competitive analysis for online routing problems with \textit{capacity restrictions}. The reason is that the service times of the tours formed for each capacitated vehicle by nontrivial policies, such as \textit{Reopt}, usually have a nonzero length and are highly interdependent, which makes a formal analysis very hard. Therefore, most existing studies analyze variants of the `first come first served' policy to leverage the methods of the queuing theory \citep[cf. e.g.,][]{ Bertsimas1993}. An exception is a prominent article by \citet{Jaillet2008}, who studied \textit{the online TSP with precedence and capacity constraints (OTSP-PC)}. \citet{Jaillet2008} examine a problem extension with several salesmen, propose online algorithms with the best-possible strict competitive ratio for several problem variants, perform competitive analysis under resource augmentation, and, for specific problem variants and under several additional assumptions, perform probabilistic asymptotic analysis of online algorithms for a renewal process of arriving requests and an arrival process with order statistics property. However, OTSP-PC studied by \citet{Jaillet2008}  \textit{differs} from OOBSRP as it allows \textit{splitting} of orders (called requests) among several tours, whereas OOBSRP does not. Consequently, once the first item of an order (called \textit{city of a request} in OTSP-PC) is collected, a commitment to collect \textit{all} the remaining items of this order within the current tour emerges in OOBSRP (see Example~\ref{ex:example1} of Section~\ref{sec:abs_CR}). This is nontrivial and changes the character of the problem completely. Overall, the idea of a probabilistic asymptotic analysis of Reopt as well as the nature of examined arrival processes in this paper were inspired by \citet{Jaillet2008}. However, this paper's methodology  is \textit{distinct} in several essential ways. For instance, \citet{Jaillet2008} rely on a number of additional assumptions in their probabilistic analysis that  do not apply to OOBSRP and are not used in this paper, such as Euclidean space, one item per order, and -- in case of a renewal process for order arrivals -- infinitely growing cart capacity (called server capacity) with the number of orders $n$. This radically changes the mechanics of the proofs and, as in Theorem~\ref{th:stoch_optimality}, requires new proof techniques.

To sum up, the optimality analysis of reoptimization for OOBSRP, which combines the features of routing (online TSP) and batching (OBP) has not been performed so far. Although the results of \textit{Reopt}  in online TSP and OBP are quite promising, it is absolutely unclear, whether it retains its good performance in the more general case of OOBSRP. 

Furthermore, this paper is the first to prove almost sure asymptotic optimality of a reoptimization algorithm in case of a fixed server capacity for any capacitated online routing problem, under a renewal process model for arrival times.

\subsection{Contribution}\label{sec:contribution}

OOBSRP, which describes picking operations in e-commerce, is a central problem in the warehousing literature. However, we still do not know the performance gap to optimality for \textit{any} OOBSRP policy. 
We aim to close this gap for one of the most prominent dynamic algorithms, namely  
immediate reoptimization \textit{(Reopt)}, which reoptimizes the current solution each time a new order arrives. Therefore, we establish \textit{analytical} performance bounds by means of \textit{worst-case and probabilistic analysis}, and \textit{empirically} validate the \textit{almost optimal} performance of \textit{Reopt}, which is indicated by our analytical findings. In particular,

\begin{itemize}
\item We analyze OOBSRP for two widespread technologies: manual pushcarts and robotic picking carts. 
Since in the former case, transmitting instructions to the picker 
is only possible at the depot in some types of warehouses \cite[cf.][]{Chen2010}, we also examine a so-called \textit{non-interventionist} version of \textit{Reopt} for manual pushcarts. \textit{Non-interventionist} \textit{Reopt} does not allow the modification of started batches, therefore reoptimization only takes place when the picker is at the depot. 
\item Our main result states that \textit{Reopt} is \textit{almost surely asymptotically optimal} under very general stochastic assumptions in all the examined picking systems. 
Stated differently, in any sufficiently large instance, the result of \textit{Reopt} coincides, with probability one, with the outcome of the optimal policy for this instance and cannot be improved. This is the case when the order arrival times can be modeled as so-called order statistics, or as a Poisson process with sufficiently small rates. In the case of a Poisson process, we also show that \textit{Reopt} achieves a slightly weaker notion of optimality when the order arrival rate is sufficiently large.  
\item In the subsequent analysis, we drop stochastic assumptions and examine the performance of \textit{Reopt} in \textit{any} large instance as well as in instances of arbitrary sizes. Among others, we prove that the worst-case performance of \textit{Reopt} compared to the optimum (\textit{competitive ratio}) approaches 2 in all studied picking systems, 
In other words, no policy can improve upon the results of \textit{Reopt} by more than 50\% in such instances.
\item Our experiments, which feature typical warehouse settings and order arrival rates, reveal \textit{almost optimal} performance of \textit{Reopt} already in moderate-sized instances.
\end{itemize}

Our analysis offers significant insights into long-standing discussions within the warehousing literature. The first pertains to the merits of \textit{waiting}, wherein the picker interrupts picking to wait in the depot (or field) for additional orders, in the hope for the arrival of better-matching orders for the next batch \citep{Bukchin2012, Pardo2023}. Waiting is an integral part of prominent fixed-time-window batching and variable-time-window batching policies \citep{Pardo2023, Vannieuwenhuyse2009}. The second involves \textit{anticipation}, which seeks to predict upcoming order properties based on historical data \citep{Ulmer2017}. In prevalent designs, anticipatory algorithms \textit{reserve} slots in batches for future orders, resulting in the under-utilization of the cart for some time and potentially leaving out well-matching orders. Both waiting and anticipation incur immediate costs for uncertain future benefits.
The performance gaps of \textit{Reopt}, which eschews both \textit{waiting} and \textit{anticipation}, establish a clear benchmark for assessing the benefits of these concepts. In scenarios where \textit{Reopt} closely approaches complete-information optimality (which is a strong concept of optimality, cf. Section~\ref{sec:problem}), the introduction of waiting or anticipation cannot significantly improve the results and could potentially lead to deterioration. 

We emphasize the novelty of the performed probabilistic analysis of \textit{Reopt}. This analysis is one of the few research endeavors to scrutinize optimality gaps of a nontrivial online policy for a variant of the capacitated vehicle routing problem, diverging from the conventional `first come first served' approach, and assuming realistic capacity restrictions (see the discussion in Section~\ref{sec:literature}).

We proceed as follows. Section~\ref{sec:problem} states OOBSRP formally, Section~\ref{sec:comp_ratio} presents theoretical analysis of \textit{Reopt} for OOBSRP, and Section~\ref{sec:experiments} reports on computational experiments. We conclude with a discussion in Section~\ref{sec:conclusion}.

\section{Problem description}\label{sec:problem}

\begin{figure}[htbp]
    \centering
    \includegraphics[scale=0.45]{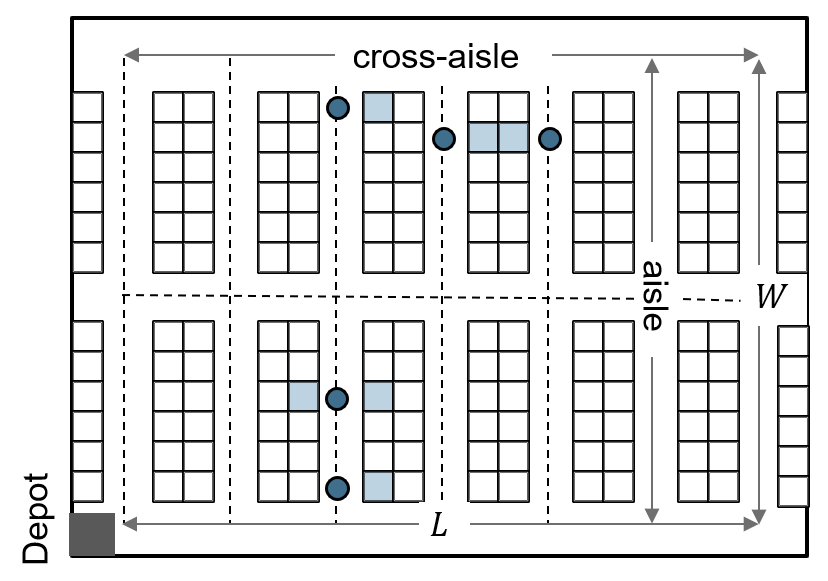}
    \caption{A typical warehouse layout. \label{fig:warehouse} \\ \scriptsize{Note. A warehouse of length $L$ and width $W$ with three cross-aisles and seven aisles. Circles mark access points to pick highlighted items.}}
    \label{fig:warehouse}
    \vspace{1ex}
\end{figure}
 
A typical warehouse (see Figure~\ref{fig:warehouse}) resembles a rectilinear grid of length $L$ and width $W$, which has $a\geq 1$ equidistantly positioned parallel aisles and $b\geq 2$ equidistantly positioned cross-aisles. The depot is located at some fixed location $l_d$ in the grid. The picker moves in a rectilinear way through aisles and cross-aisles, which results in a particular distance metric $d(~)$. As explained in Figure~\ref{fig:warehouse}, the picker can access items only from 
the aisles, in particular, no items are located in cross-aisles.

Orders $(o_1, o_2, \ldots,o_j, \ldots)$ arrive dynamically over time and each order $o_j$ is associated with some number $k(j)\in \mathbb{N}$ of \textit{picking locations} in the warehouse, where the ordered \textit{items} have to be retrieved. Following the hierarchical planning process in warehouses, OOBSRP assumes that the picking location of each ordered item has been already determined, therefore, we use these terms interchangeably. Recall that the cart of the picker consists of $c\in \mathbb{N}$ 
bins, so that, to save on costs, she collects up to $c$ orders at a time during one cart tour in a so-called \textit{batch}. Orders cannot be split between different bins and each bin may contain only one order at a time. Parameter $c$ is called \textit{batching capacity}. 

We examine \textit{picking systems} with two different types of carts: manual pushcarts and robotic carts; we will denote these carts as \textit{pushcart} and \textit{robot}, respectively, in the following. The main difference between them is that the picker has to bring the pushcart to the depot each time after a batch has been collected; whereas the robot can drive there autonomously and the picker can proceed to pick the items `in the field' into the timely arrived new robot. 

OOBSRP considers one picking zone of a single picker. Therefore, when we talk about the warehouse layout below, we refer to the layout of a single zone of the warehouse.

We make the following assumptions:
\begin{itemize}
\item \textit{W.l.o.g.}, we assume that the picker travels at unit speed, so that we talk about the distances and picker travel times interchangeably. In our theoretical analysis, we set the picking times of retrieving an item from the shelf and placing it in the cart to zero for simplicity.  
Later on, in the empirical analysis of Section~\ref{sec:experiments}, we consider non-zero picking times. 
\item Following the literature on robotic carts \citep[e.g., ][]{Loeffler2022,Zulj2022}, we assume that an empty robot is immediately available for the picker after she has completed a batch. Indeed, two robots per picker are usually sufficient to avoid picker waiting \citep[cf.][]{Loeffler2022}.  
\end{itemize}

Before stating OOBSRP in Section~\ref{sec:onlineproblem}, we introduce the \textit{offline} problem variant with complete information (OBSRP$^*$) in Section~\ref{sec:offlineproblem}. Figure~\ref{fig:Ill_ex} states an illustrative example that we use throughout this section.
\subsection{Offline problem variant with complete information: OBSRP$^*$}\label{sec:offlineproblem}

\begin{figure}
\centering
    \begin{subfigure}[b]{0.39\textwidth}
\includegraphics[width=\textwidth]{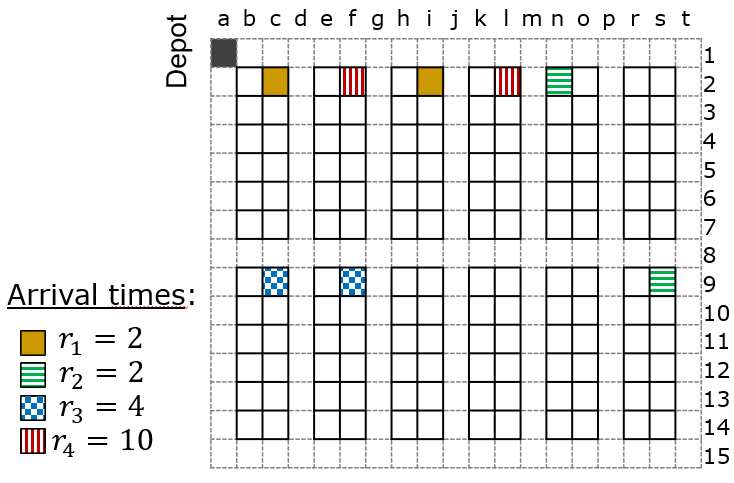}
        \caption*{\footnotesize{(a) Instance $I$}}
        \label{fig:pushcart}
    \end{subfigure}
    ~ 
    \begin{subfigure}[b]{0.28\textwidth}
\includegraphics[width=\textwidth]{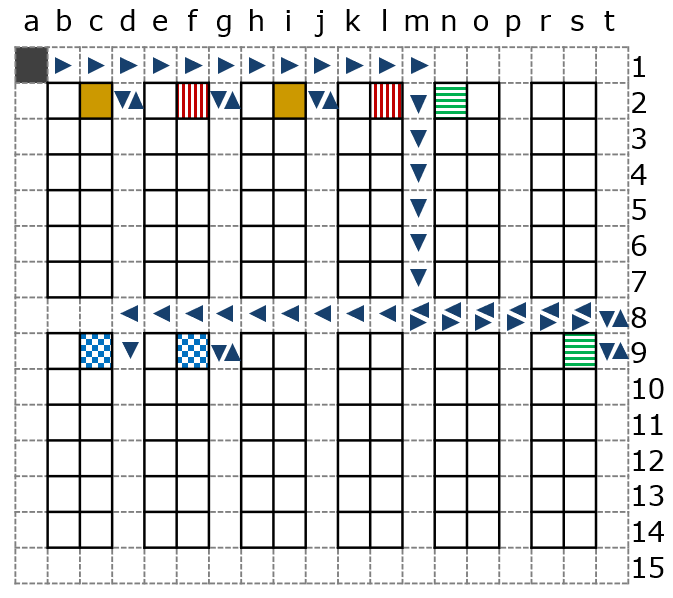}
        \caption*{\footnotesize{(b) $CIOPT$ Picker route with \textit{robot}}}
        \label{fig:pushcart}
    \end{subfigure}
    ~ 
    \begin{subfigure}[b]{0.28\textwidth}
        \includegraphics[width=\textwidth]{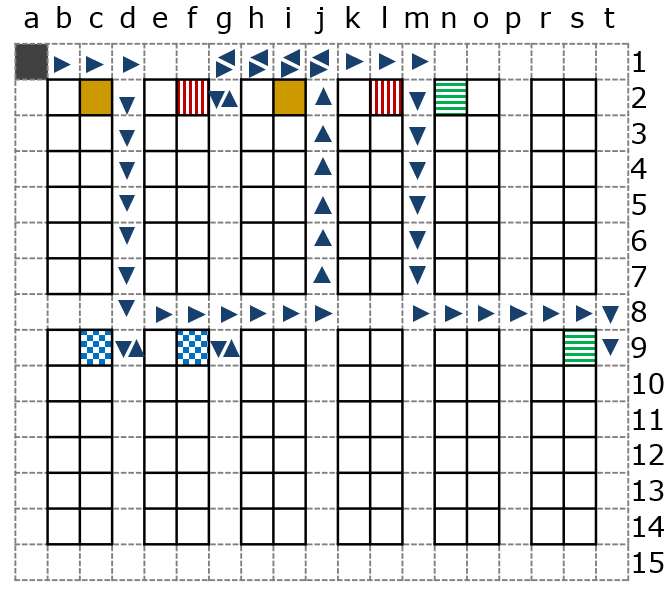}
        \caption*{\footnotesize{(c) \textit{Reopt} Picker route with \textit{robot}}}
        \label{fig:robot}
    \end{subfigure}
\newline   
    \caption{Illustrative example. \hspace{10cm}  \newline
    \scriptsize{ Note. Items of all ($n=4$) orders are marked in different colors.}} \label{fig:Ill_ex}
\end{figure}

An \textit{instance} $I$ of OBSRP$^*$ \textit{with a specific type of cart} has the following input:
\begin{itemize}
\item Infrastructure parameters, such as the location of the depot $l_d$, width $W$, length $L$, the number of aisles $a$ and cross-aisles $b$,  as well as the batching capacity $c$;
\item the \textit{sequence of} dynamically arriving \textit{orders } $o=(o_1,...,o_n)$, where $o_j:=\{s_j^1,...,s_j^{k(j)}\}$ refers to the set of  $k(j)\geq 1$ picking locations for the $j^{\text{th}}$ order; 
\item the respective \textit{sequence of arrival times} $r=(r_1, ...,r_n)$, with $r_j$ being the arrival time of $o_j$ and $r_1\leq ... \leq r_n$. Observe that in OBSRP$^*$, where all the information is known in advance, arrival times can be simply interpreted as \textit{release times} for picking the items of the corresponding orders. 

\end{itemize}

The objective is to minimize the \textit{total completion time} of picking, which describes the picker's working time. In the system with a pushcart, the picker has to return the pushcart to the depot in the end. On the contrary, in the system with a robot, the total completion time describes the time passed until the picker picks the \textit{last item} of the last batch. 
Given this objective, an optimal solution of OBSRP$^*$ for both cases of a pushcart and a robot is uniquely defined by:
\begin{itemize}
\item a mutually disjoint partition of the orders into batches: $\{o_1, ..., o_n\}=S_1 \cup S_2 \cup\ldots \cup S_f$, \\ $S_l\cap S_k=\emptyset \ \forall k,l \in \{1,...,f\}$; each batch contains at most $c$ orders.
\item a sequence of these batches $\pi^{\text{batches}}$, which determines the order in which the batches are processed by the picker.
\item for each batch $S_l$, a permutation of the picking locations of all the included orders, which determines in which order the picker collects the items.
\end{itemize}

Indeed, in an optimal picking schedule, the picker will pick each item at the earliest possible time (which accounts for the arrival times of the orders) given the sequence of batches $\pi^{\text{batches}}$ and the visiting sequences of picking locations for each batch. 
We denote the optimal objective value of instance $I$ as \textit{\underline{c}omplete \underline{i}nformation \underline{opt}imum} $CIOPT(I)$ and a respective optimal solution as $CIOPT$.

\begin{table}
\centering
\caption{Algorithms and their solutions for the illustrative instance $I$ from Figure~\ref{fig:Ill_ex}}\label{tab:Solutions_of_ex}
\scriptsize{
\begin{tabular}{lcll}
\toprule
\multirow{2}{*}{~Algorithm}& Objective & Sequence of batches & \multirow{2}{*}{Picker route$^\text{ii)}$ $\pi$}\\
 & value& $\pi^{\text{batches}}=(S_1,S_2,\ldots)$& \\
\midrule
\multicolumn{4}{l}{\underline{Case of a \textit{pushcart}}}\\
~$CIOPT^\text{i)}$& $88$ & $(\{o_1,o_3\}\{o_2,o_4\})$ & $(a1,d2,d9,g9,j2,a1,g2,t9,m2,m2,a1)$ \\
~Non-interventionist \textit{Reopt} & $100$ &  $(\{o_1,o_2\},\{o_3,o_4\})$  & $(a1,d2,j2,t9,m2,a1,d9,g9,m2,g2,a1)$, \\
~Interventionist \textit{Reopt}& $90$ & $(\{o_1,o_4\}\{o_2,o_3\})$ & $(a1,\mathbf{c1},d2,\textbf{g1},g2,m2,j2,a1,m2,t9,g9,d9,a1)$ \\
\multicolumn{4}{l}{}\\
\multicolumn{4}{l}{\underline{Case of a \textit{robot}}}\\
~$CIOPT$& $52$ &$(\{o_1,o_4\},\{o_2,o_3\})$ & $(a1,d2,g2,j2,m2,m2,t9,g9,d9)$ \\
~Interventionist \textit{Reopt}& $54$ & $(\{o_1,o_3\},\{o_2,o_4\})$  & $(a1,\textbf{c1},d2,\textbf{d6},d9,g9,j2,g2,m2,m2,t9)$ \\
\midrule[\heavyrulewidth]
\multicolumn{4}{l}{\footnotesize{\textit{Note.} $^\text{i)}$  $CIOPT$ refers to an exact algorithm for the complete-information counterpart of instance $I$.}}\\
\multicolumn{4}{l}{$^\text{ii)}$ The picker positions at reoptimization events of the interventionist \textit{Reopt} are marked in bold.}
  \end{tabular}}
\end{table}

Table~\ref{tab:Solutions_of_ex} illustrates $CIOPT(I)$ 
for the instance $I$ with $n=4$ orders provided in Figure~\ref{fig:Ill_ex}a in the case of a pushcart and a robot, respectively.

Note that in the optimal solutions for both cases, the picker leaves the depot towards the first picking location even before the arrival of the respective order, since the arrival times and picking locations of all the orders are known in advance. Also observe that grid cells in Figure~\ref{fig:Ill_ex} have a unit length, e.g., $d(a1,d2)=4$ as the picker has to traverse four units (cell lengths) between the cell centers of the depot $a1$ and $d2$.

In an optimal solution of OBSRP$^*$ with a pushcart, the picker arrives at each picking location after the arrival of the corresponding order. The first batch consists of two orders ($S_1=\{o_1, o_3\}$) and is completed in $(4+7+5+10+10)=36$ time units, after the picker's return to the depot. The second batch $S_2=\{o_2, o_4\}$ is picked in $(7+19+13+0+13)=52$ time units. Overall, $CIOPT$ requires $(36+52)=88$ time units. 

Observe that in $CIOPT$ for OBSRP$^*$ with a robot (c.f. Figure~\ref{fig:Ill_ex}b), the picker waits for the arrival of order $o_4$ for one time unit at location $g2$ while performing the first batch. Overall, batch $S_1=\{o_1,o_4\}$ is picked in $(4+5+[1]+5+5)=20$ time units. After having completed the first batch, the loaded robot returns to the depot autonomously, a new empty robot is immediately available, and the picker continues with the next batch $S_2=\{o_2,o_3\}$ from her current location $m2$. $CIOPT$ requires $(20+32)=52$ time units. 
\subsection{Online problem and the reoptimization policy}\label{sec:onlineproblem}

In the \textit{online} problem, OOBSRP, neither the number of orders $n$, nor the arrival times of orders are known ahead of time, and the characteristics of each order $o_j$ are revealed at its arrival $r_j$. Therefore, online algorithms represent \textit{policies}, which prescribe a decision given the current state and available information. We analyze the online algorithm \textit{Reopt}, which optimizes picking operations based on the available information without any anticipation of future orders. 

In traditional warehouses with a pushcart, the picker can receive instructions only at the depot, for example, because the cart- and picker equipment is order-specific for each batch, so that an already commenced batch cannot be replanned. Therefore, we differentiate between \textit{non-interventionist Reopt} and \textit{interventionist Reopt}. In \textit{non-interventionist} \textit{Reopt}, reoptimization occurs at each return of the picker to the depot provided newly arrived orders are available. If all the available orders have been completed, the picker simply stays at the depot until further orders arrive. Let reoptimization take place at some time $t$, then it amounts to solving the \textit{complete-information} problem OBSRP$^*$ for the orders, which arrived before $t$ \textit{and} have not been picked so far; obviously, these orders are treated as immediately available. 

To the contrary, in the \textit{interventionist Reopt}, reoptimization is performed at each arrival $r_j$ of a new order $o_j$. 
Also in this case, the planner solves a variant of OBSRP$^*$ for the already arrived, but not yet processed orders; however, she takes the following aspects into account: \textit{i)} if some bins of the cart contain picked items, the respective orders must be completed as part of the \textit{current} batch; \textit{ii)}  in the current batch, empty bins are free to be (re)assigned to any available and not yet commenced order, \textit{iii)} the picker route starts from her current position in the warehouse. 

To sum up, we examine three picking systems for OOBSRP, each featuring a cart technology and the applied \textit{Reopt} algorithm: One with a pushcart and a noninterventionist \textit{Reopt} policy (\textit{Pcart-N}), one with a pushcart and interventionist \textit{Reopt} (\textit{Pcart}), and one with a robot and interventionist \textit{Reopt} (\textit{Robot}).

Table~\ref{tab:Solutions_of_ex} illustrates the objective values $Reopt(I)$ of \textit{Reopt} for the instance $I$ from Figure~\ref{fig:Ill_ex} in the three considered picking systems: \textit{Pcart-N}, \textit{Pcart}, and \textit{Robot}. In all three systems, the picker stays at the depot until the first two orders arrive at time $r_1=r_2=2$. 

In \textit{Pcart-N}, at the time $t=2$ of the first (re)optimization, only orders $o_1$ and $o_2$ are available and the batch $S_1=\{o_1, o_2\}$ with the completion time of $(4+8+16+13+13)=54$ is formed.  Although new orders arrive at times 4 and 10, no replanning can take place until the picker returns to the depot. The next reoptimization occurs at time $(54+2)=56$ and the remaining orders are assigned to the next batch. Altogether, $Reopt(I)=56+44=100$ for \textit{Pcart-N}. 

In contrast, in \textit{Pcart} and \textit{Robot}, interventionist \textit{Reopt} reoptimizes the current solution at each arrival of a new order. In both systems, at $r_3=4$ the picker is at decision point $c1$, and the planned routes for the batch $S_1=\{o_1,o_2\}$ are $(a1,d2,m2,t9,j2,a1)$ and $(a1,d2,j2,m2,t9)$ in the case of \textit{Pcart} and \textit{Robot}, respectively. 
For \textit{Pcart}, the new sequence of batches is $(\{o_1\},\{o_2,o_3\})$ and the new route is $(a1, \textbf{c1}, d2,j2,a1,m2,t9,g9,d9,a1)$. This route is interrupted at the arrival of $o_4$ at $r_4=10$, when the picker is at $g1$, resulting in the final solution of Table~\ref{tab:Solutions_of_ex} with $Reopt(I)=90$.
For \textit{Robot}, the batching is first changed to $(\{o_1,o_3\}\{o_2\})$ and the route is changed to $(a1,\textbf{c1},d2,d9,g9,j2,m2,t9)$. Order $o_4$ arrives when the picker is at position $d6$ having picked the first item of $o_1$. Therefore, at the next reoptimization at time $r_4=10$, one bin in the current batch is reserved for order $o_1$. So, the new batching is $(\{o_1,o_3\},\{o_2,o_4\})$ and results in $Reopt(I)=33+21=54$ for \textit{Robot} (c.f. Figure~\ref{fig:Ill_ex}c).

\section{Extended competitive analysis for \textit{Reopt} in OOBSRP}\label{sec:comp_ratio}

In this section, we perform competitive analysis for \textit{Reopt}. In competitive analysis, the result of some algorithm $ALG(I)$ for instance $I$ is compared to the best possible online policy for this instance. Obviously, the best possible online policy is that of an oracle, which perfectly foresees the arrival time and the composition of each order, and the result of this policy is $CIOPT(I)$ \citep[cf. ][]{woegingerandfiat1998}.

We apply three different ratios to highlight the various aspects of OOBSRP performance. We start in Section~\ref{sec:asymp_opt} with the \textit{almost surely (a.s.) asymptotical competitive ratio} $\sigma^{a.s.}_{asymp.}(ALG)$, which assumes a probabilistic lens and disregards improbable events. We say that $\sigma^{a.s.}_{asymp.}(ALG)=\alpha$, if:
\begin{align}
    \mathbb{P}(\lim_{n \rightarrow \infty} \frac{ALG(I(n)^{rand})}{ CIOPT(I(n)^{rand})} = \alpha) = 1,
\end{align}
where $I(n)^{rand}$ is a random instance with $n$ orders following given stochastic assumptions. Note that $\sigma^{a.s.}_{asymp.}(ALG)$ strongly depends on the stochastic assumptions under consideration.

To analyze \textit{general} large instances in Section~\ref{sec:asy_comp_ratio}, we turn to the \textit{asymptotic competitive ratio $\sigma_{asymp.}(ALG)$}. We call algorithm $ALG$ \textit{asymptotically} $\alpha$-\textit{competitive} if, for any instance $I$, there exists a constant $const$, which is independent of the number and characteristics of the orders in $I$, but may depend on the warehouse geometry, such that: 
\begin{align}
&&& ALG(I)\leq \alpha \cdot  CIOPT(I) + const& \label{eq:asy_def}
\end{align}
The \textit{asymptotic competitive ratio} $\sigma_{asymp.}(ALG)$ is then defined as the infimum over all constants $\alpha$, such that $ALG$ is asymptotically $\alpha$-competitive.

Finally, to include small instances in our analysis (Section~\ref{sec:abs_CR}), which may display anomalies, we use the \textit{strict competitive ratio $\sigma(ALG)$}, which is the worst-case ratio between $ALG(I)$ and $CIOPT(I)$ over \textit{all} possible instances $I$ of the problem:
\begin{align}
&&& \sigma(ALG)=\sup_{I}\frac{ALG(I)}{ CIOPT(I)}.&
\end{align}
Observe that the strict competitive ratio derives from the asymptotic competitive ratio by forcing the constant $const$ to be zero.

We begin by stating the key properties of OOBSRP in Section~\ref{sec:basics}. Afterward, Section~\ref{sec:asymp_opt} establishes almost sure asymptotic optimality of \textit{Reopt} in all the examined systems under general stochastic conditions. Section~\ref{sec:asy_comp_ratio} and Section~\ref{sec:abs_CR} examine the asymptotic and the strict competitive ratios of \textit{Reopt}, respectively.


\subsection{Properties of OOBSRP} \label{sec:basics}
The results of this section are used throughout the following discussion. We state bounds for a warehouse traversal in Lemma~\ref{prop:cover} as well as for the results of $Reopt$  and $CIOPT$ in Lemmas~\ref{prop:UB1}--\ref{prop:simple}, respectively. We  conclude by showing asymptotic optimality of \textit{Reopt} for a special OOBSRP case, where each order contains exactly \textit{one} item.

Let us denote an OOBSRP instance with $n$ orders as $I(n)$. We denote as $I(n)^{r=0}$ the variant of some given instance $I(n)$, in which all the orders are instantly available, i.e., with $r_1=r_2=\ldots=r_n=0$. 

\begin{lemma}[Upper bound for the warehouse traversal]\label{prop:cover}
Notwithstanding the cart technology -- a pushcart or a robot -- and given an arbitrary starting position of the picker, the shortest time needed to visit all the picking locations in the warehouse and return to any given location is bounded from above by the following constant $u$. The bound $u$ is tight.
\begin{align}
    u=(a+1)W+2L.
\end{align}
\end{lemma}
\begin{proof} See Electronic Companion~\ref{sec:Proof_Lemma_warehouse_traversal}.
\end{proof}

\begin{lemma}[Upper bound for $Reopt(I)$]
\label{prop:UB1}
In \textit{Pcart-N}, \textit{Pcart}, and \textit{Robot}, the results of the respective \textit{Reopt} policy for some instance $I(n)$ cannot be worse than the following bound:
\begin{align}
Reopt(I(n))\leq r_n+ CIOPT(I(n)^{r=0})+u  \label{eq:UB_reopt}
\end{align}
\end{lemma}
\begin{proof}
At time $r_n$, when the last order arrives, \textit{Reopt} reoptimizes having the complete information and requires no more time than the following simple policy: Complete the currently open batch and return to the depot in at most $u$ time (cf. Lemma~\ref{prop:cover}), then follow the route of $CIOPT$ for the instance $I(n)^{r=0}$ to collect the remaining orders.  
 \end{proof}

For \textit{Pcart} and \textit{Pcart-N}, let us define the \textit{makespan} $\mathcal{M}(o_j)$ of order $o_j$ as the length of the shortest route that starts at the depot, covers all items $s_j^{i}\in o_j$, and ends at the depot. Similarly, for $Robot$, we define the \textit{makespan} $\mathcal{M}(o_j)$ of order $o_j$ as the length of the shortest route that starts in any picking location $s_j^{s}\in o_j$, covers all picking locations in $o_j$, and ends at any arbitrary picking location $s_j^{t}\in o_j$.

\begin{lemma}
\label{prop:LB1}
In \textit{Pcart-N}, \textit{Pcart}, and \textit{Robot}, the complete information optimum for any instance $I(n)^{r=0}$, for which all orders are initially available, cannot be better than the following bound:
\begin{align}
 CIOPT(I(n)^{r=0})\geq {\frac{\sum_{i=1}^{n}{\mathcal{M}(o_i)}}{c}}, \label{eq:LB_offline}
\end{align}
\end{lemma}
\begin{proof}
This bound assumes the ideal case in which each batch can be fully packed with $c$ perfectly matching orders. These orders are then collected simultaneously during the picking process at the walking cost of only one of these orders, $\mathcal{M}(o_j)$, for $o_j$ in the batch. In the case of \textit{Robot}, it further ignores possible walking time between the items of subsequent \textit{distinct} batches. 
\end{proof}

The following Lemma~\ref{prop:simple} summarizes some important straightforward relations, which we state without a proof. 

\begin{lemma}[Lower bounds for $CIOPT(I)$]
\label{prop:simple}
In \textit{Pcart-N}, \textit{Pcart}, and \textit{Robot}, and any instance $I(n)$:
\begin{align}
CIOPT(I(n))&\geq r_n &&\label{eq:simple_release}\\
CIOPT(I(n))&\geq CIOPT(I(n)^{r=0})\\
&\geq CIOPT(J^{r=0})&& \forall J\subseteq I(n)  \label{eq:subsets}
\end{align}
\end{lemma}

\begin{proposition} \label{prop:one_item_order}
In \textit{Robot}, in the special case of OOBSRP with single-item orders, i.e., $k(o_j)=1 \ \forall j \in \{1, \ldots, n\}$, \textit{Reopt} is \textit{asymptotically optimal}, i.e., $\sigma_{asymp.}(Reopt)=1$.
\end{proposition} 
\begin{proof} Observe that \textit{Robot} with single-item orders is equivalent to the online nomadic TSP. So, at the arrival of the last order at time $r_n$, \textit{Reopt} reoptimizes having the complete information and collects the remaining orders in no more time than that required by the following simple policy: Traverse the warehouse in S-shape motion once. This traversal requires a constant amount of time, dependent only on the warehouse geometry. Using (\ref{eq:simple_release}) and the definition of the asymptotic convergence in (\ref{eq:asy_def}), we complete the proof. 
\end{proof}

\subsection{Asymptotic optimality of \textit{Reopt}} \label{sec:asymp_opt}

In this section, we show that, under general stochastic assumptions, when the number of incoming orders is sufficiently large, \textit{Reopt} is asymptotically optimal with a probability of one (\textit{almost surely}) in all the examined systems: \textit{Pcart-N}, \textit{Pcart}, and \textit{Robot}.

In a probabilistic analysis, we consider \textit{random} instances, and such an instance with $n\in\mathbb{N}$ orders is denoted as $I(n)^{rand}$. We denote the involved random variables in capital letters, e.g., orders $O=(O_1, O_2, \ldots, O_n)$ and arrivals $R=(R_1,R_2,..,R_n)$. 
We assume that \textit{distinct} customer orders $O_j$ in $O$ are \textit{independently identically distributed (i.i.d.)}. For mathematical completeness, we redefine $O_j$ as a multivariate random variable $O_j=(K_j,S_j^1,...,S_j^{K_j})$ with $K_j$ picking locations $S_j^1,\ldots,S_j^{K_j}$. Further, in case of \textit{Robot}, we exclude the trivial case of single-item orders (cf. Proposition~\ref{prop:one_item_order}) and assume $\mathbb{E}[ K_1 ]>1$. 
Observe that the specified stochastic assumptions explicitly allow product items $S_j^i$ and $S_j^l$ of the \textit{same} order $O_j$ to be correlated, such as phones and matching protective cases, which are usually ordered together. 

We assume that the random sequences $R$ and $O$ are independent. 
In the following, we will discuss two common stochastic models for the arrival time sequence $R$. 

We begin by examining the order statistics property of arrival times in Section~\ref{sec:orderstat}, an order arrival model  suggested in the context of online routing problems by \citet{Jaillet2008}. This model considers a fixed number $n$ of orders, with arrival times that are \textit{independent and identically distributed (i.i.d.)} according to a fixed but \textit{arbitrary} distribution. The \textit{i.i.d.} assumption is natural in e-commerce warehousing, since customers act independently. The distribution can be arbitrary (see Figure~\ref{fig:ord_stat}), e.g., uniform, single-peaked, skewed, multimodal (the latter applies, for instance, when most orders arrive during lunch breaks and evenings). Order statistics arrival model can have both \textit{finite} and \textit{infinite} support; in the former case, the orders arrive within a finite time interval, like a shift or a day, while in the latter, order arrival times are unrestricted. 
 
For an example of the latter case, consider picking a stock of 
$n$ limited-edition cell phones from a warehouse that exclusively contains those items. The picking process concludes when all $n$ items have been picked. 

The analysis in Section~\ref{sec:poisson}, where order arrivals follow a homogeneous Poisson process, builds upon the results of Section~\ref{sec:orderstat}. 
We proceed with a formal definition of the two models and the analytical results in Sections~\ref{sec:orderstat} and Section~\ref{sec:poisson}.

\begin{figure}
    \begin{subfigure}[b]{0.31\textwidth}
        \includegraphics[width=\textwidth]{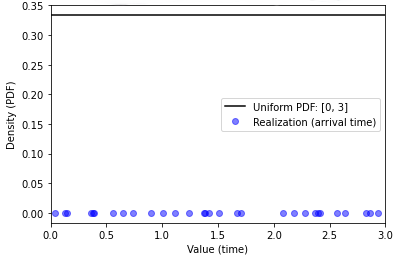}
        \caption*{\scriptsize{(a) Uniform distribution}} 
        \label{fig:pushcart}
    \end{subfigure}
    ~ 
    \begin{subfigure}[b]{0.3\textwidth}
        \includegraphics[width=\textwidth]{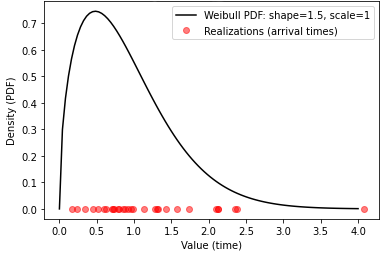}
        \caption*{\scriptsize{(b) Weibull distribution}}
        \label{fig:robot}
    \end{subfigure}
    \begin{subfigure}[b]{0.285\textwidth}
        \includegraphics[width=\textwidth]{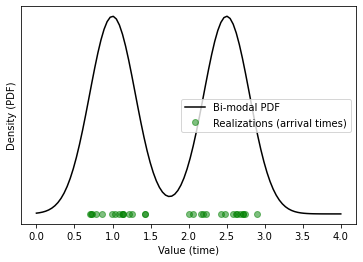}
        \caption*{\scriptsize{(c) Bi-modal distribution}}
        \label{fig:robot}
    \end{subfigure}
    \caption{Realizations of 30 order arrival times with the order statistics model for different distributions\newline \scriptsize{\textit{Note.}The three graphs show the probability distribution function (PDF) of the respective distributions (line) and the dots on the $x$-axis represent the realizations of the arrival times}}\label{fig:ord_stat}
\end{figure}

\subsubsection{Case of the order statistics property of arrival times} \label{sec:orderstat}

The \textit{order statistics property of arrival times} can be defined as follows. 
The arrival times of the orders are i.i.d. realizations $Y_j, j\in \{1,...,n\},$ of a generic random variable $Y\geq 0$ with an arbitrary, given distribution and mean $\mu_Y<\infty$. To derive the arrival times of the orders, we sort the realizations of $Y$, the $j^{th}$ arrival time is the $j^{th}$ order statistic: $R_j=Y_{(j)}$,  $Y_{(1)}\leq Y_{(2)}\leq ... \leq Y_{(n)}$.

Theorem \ref{th:asymp_opt_makesp} states one of the main results of this work.

\begin{theorem}
\label{th:asymp_opt_makesp}
In \textit{Pcart-N}, \textit{Pcart}, and \textit{Robot}, given the specified stochastic assumptions, including the order statistics property of arrival times:
\begin{align}
  &&& \lim_{n\rightarrow \infty } \ \frac{Reopt(I(n)^{rand})}{ CIOPT(I(n)^{rand})}=1  \ \ \text{ a.s.} \label{eq:asymp_opt}
\end{align}
\end{theorem}
\begin{proof}
Following a similar line of proof as \cite{Jaillet2008}, we first recall two relations from Lemmas~\ref{prop:simple} and~\ref{prop:UB1}, respectively:
\begin{align}
&&CIOPT(I(n)^{rand})  &\geq  CIOPT(I(n)^{rand,r=0})  \label{eq:Opt_bound} \\ 
&&  Reopt(I(n)^{rand})  &\leq CIOPT(I(n)^{rand,r=0}) + R_n + u \label{eq:Reopt_bound} 
\end{align}
Together, they imply
\begin{align}
&&&\frac{Reopt(I(n)^{rand})}{ CIOPT(I(n)^{rand})} 
 \leq  \ 1 + \frac{n}{CIOPT(I(n)^{rand,r=0})} \cdot \frac{R_n+u}{n} \label{eq:ratio_LB}
\end{align}
Since $\mu_Y<\infty$, Lemma 6 of \cite{Jaillet2008} can be applied:
\begin{align}
    &&& \lim_{n\rightarrow \infty}\frac{R_n}{n}+\frac{u}{n}=0 \qquad \text{a.s.} \label{eq:G_conv}
\end{align}
At this point, the proof of \cite{Jaillet2008} developed for the capacitated online TSP is not applicable and we proceed in a unique way. To show that $\frac{n}{CIOPT(I(n)^{rand,r=0})}$ is bounded from above by a constant (which does not depend on $n$), we use the lower bound from Lemma~\ref{prop:LB1}:
\begin{align}
&&    \frac{ CIOPT(I(n)^{rand,r=0})}{n} 
& \geq \frac{1}{c} \cdot  \frac{\sum\limits_{j=1}^n \mathcal{M}(O_j)}{n} \ \  \label{eq:average_tot_makespan}
\end{align}
Recall that essentially makespan $\mathcal{M}(O_j)$ of order $O_j$ is the shortest picker route to collect all items in $O_j$. Given the stochastic assumptions, $\mathcal{M}(O_j), j \in \{1,...,n\},$ are i.i.d.. So by the Strong Law of Large Numbers:
\begin{align}
&&& \lim_{n\rightarrow \infty} \frac{1}{c} \cdot \frac{\sum\limits_{j =1}^n \mathcal{M}(O_j)}{n} = \frac{1}{c} \cdot \mathbb{E}[\mathcal{M}(O_1)]  >0 \qquad \text{ a.s.}   \label{eq:as_conv_UB}
\end{align}
We only consider the case of a strictly positive $\mathbb{E}[\mathcal{M}(O_1)]$. Otherwise, the problem is trivial, and \textit{Reopt} is optimal per definition in \textit{Pcart-N} and \textit{Pcart}; in \textit{Robot}, the problem reduces to the single-item order case for which \textit{Reopt} is asymptotically optimal by Proposition~\ref{prop:one_item_order}. 

Observe that $\frac{1}{c} \cdot \mathbb{E}[\mathcal{M}(O_1)]$ is a constant that does not depend on $n$. The convergence of (\ref{eq:as_conv_UB}) implies that the right-hand sight of (\ref{eq:average_tot_makespan}) is larger than some \textit{nonzero} constant $\zeta<\frac{1}{c}\mathbb{E}[\mathcal{M}(O_1)]$ a.s. if the number of orders $n$ is large enough, i.e.,  $\forall n>n_0 \text{ for some }n_0$:
\begin{align}
& &&& \frac{CIOPT(I(n)^{rand,r=0})}{n} \geq    \zeta  
\Leftrightarrow \   \frac{n}{CIOPT(I(n)^{rand,r=0})}\leq \frac{1}{\zeta} \qquad \text{ a.s. } \label{eq:F_bound_rev}
\end{align}
This, together with (\ref{eq:ratio_LB}) and (\ref{eq:G_conv}), completes the proof:
\begin{align}
&& \frac{Reopt(I(n)^{rand})}{ CIOPT(I(n)^{rand})} 
\leq & \  1 + \frac{n}{CIOPT(I(n)^{rand,r=0})} \cdot \frac{R_n+u}{n} \nonumber 
 \xrightarrow{n \rightarrow \infty} 1 +0 \qquad  \text{a.s.} 
\end{align} 
\end{proof}

The interpretation of Theorem~\ref{th:asymp_opt_makesp} for an order statistic with a \textit{finite}-support arrival-time distribution $Y$ such as during a shift, is quite intuitive. As the number of orders arrived during this shift increases, the picker becomes overloaded, and most orders will have to be picked \textit{after} the shift ends, regardless of the picking policy. Therefore, Reopt, which picks optimally after the end of the shift when no more orders arrive, converges to optimality. 
However, the interpretation is less straightforward for a random arrival time distribution $Y$ with \textit{infinite} support, such as a normal distribution in the limited-edition cell phone example in Section~\ref{sec:asymp_opt}. 
In this case, as the number of available products increases, $n\rightarrow \infty$, the arrival time of the last order placement, $R_n$, will extend to infinity (as for arbitrary $x<\infty$, $\mathbb{P}(R_n>x)=1-\mathbb{P}(Y\leq x)^n\rightarrow 1$ by definition of $Y$'s infinite support). This implies that the picking system remains \textit{dynamic} throughout the entire picking process. Even for these systems, Theorem~\ref{th:asymp_opt_makesp} proves that Reopt can keep up with an optimal picking strategy, with probability one. 

\subsubsection{Case of the homogeneous Poisson process for order arrivals} \label{sec:poisson}

Let random variables $X_j, j \in \mathbb{N},$ which represent the time between the $j^{th}$ and the $(j-1)^{th}$ order arrival, be i.i.d. exponentially distributed with mean $\frac{1}{\lambda}$, then the sequence of order arrival times follows a \textit{homogeneous Poisson process} with rate  $\lambda$.

A central notion for the analysis of the Poisson-arriving orders is the \textit{queue} 
$Q^{Reopt}(I(n)^{rand})$, which is the set of pending orders at time $R_n$ in instance $I(n)^{rand}$ for policy \textit{Reopt}.

We prove in Theorem~\ref{th:a.s.Poisson} that \textit{Reopt} is \textit{a.s.} asymptotically optimal if $ \lim_{n\rightarrow \infty} \frac{\vert Q^{Reopt}(I(n)^{rand})\vert}{n} = 0 $, i.e. if with probability one, the queue of pending orders does not grow as fast as the total number of received orders. Note that this condition would be necessary under the \textit{Reopt} policy to safeguard the warehousing system against potential instability.  

Because the queue length $\vert Q^{Reopt}(I(n)^{rand})\vert$ is generated by two independent sequences of i.i.d. random variables -- the picking locations and the inter-arrival times -- we believe that by the 0-1-Laws, the probability of the event $ \lim_{n\rightarrow \infty} \frac{\vert Q^{Reopt}(I(n)^{rand})\vert}{n} = 0 $ is either 1 or zero for a given rate $\lambda$, and we suspect that the probability increases monotonously in $\lambda$.  
In other words, we believe that $ \lim_{n\rightarrow \infty} \frac{\vert Q^{Reopt}(I(n)^{rand})\vert}{n} = 0 $ is true for small arrival rates until a fixed threshold $\overline{\lambda}$ (see Conjecture~\ref{con:ConjecturePoisson}).

In the case of asymptotically increasing arrival rates, the probability that \textit{Reopt} approaches optimality converges to 1 as well, as we prove in  Theorem~\ref{th:stoch_optimality}. 

Observe the \textit{stand-alone character} of this section's arguments in the warehousing research. The available queuing systems analysis, which is usually employed when order arrivals follow a Poisson process, requires i.i.d. service times of the batches. The latter is the case, for instance, for the `first come first served' policy, when orders are serviced in the sequence of their arrival. This is not the case for \textit{Reopt} and neither it is the case for $CIOPT$, where the service times of the batches are interdependent due to reoptimization. 

\begin{theorem}
\label{th:a.s.Poisson}
In \textit{Pcart-N}, \textit{Pcart},  and \textit{Robot}, if $ \lim_{n\rightarrow \infty} \frac{\vert Q^{Reopt}(I(n)^{rand})\vert}{n} = 0 $ a.s., then:
\begin{align}
&&&   \lim_{n\rightarrow \infty } \ \frac{Reopt(I(n)^{rand})}{ CIOPT(I(n)^{rand})}=1  \ \ \text{ a.s.} \label{eq:asymp_opt}
\end{align}
\end{theorem}
\begin{proof}
Exploiting the i.i.d. inter-arrival time property of the Poisson process, the proof combines the Strong Law of Large Numbers with the fact that the warehouse has bounded dimensions. We refer to the Electronic Companion~\ref{sec:Proof_Poisson_a.s.} for details. 
\end{proof}

Observe that Theorem~\ref{th:a.s.Poisson} demonstrates the asymptotic optimality of Reopt  not only for a.s. \textit{stable} picking systems, but also for certain systems with 
 high order arrival rates, where there is a non-zero probability that the queue length grows to infinity over time. For instance, it applies to picking systems with queues growing as fast as $\sqrt{n}$, which is sublinear in $n$. In summary, the statement of Theorem~\ref{th:a.s.Poisson} is quite general, extending from small to slightly too high order arrival rates. 

\begin{conjecture}\label{con:ConjecturePoisson}
In \textit{Pcart-N}, \textit{Pcart}, and \textit{Robot}, there is a rate $\tilde{\lambda}$ such that  $\mathbb{P}( \lim_{n\rightarrow \infty} \frac{\vert Q^{Reopt}(I(n)^{rand})\vert}{n} = 0) =1$ for all $\lambda<\tilde{\lambda}$. Thus, for all rates $\lambda<\tilde{\lambda}$:
\begin{align}
&&&   \lim_{n\rightarrow \infty } \ \frac{Reopt(I(n)^{rand})}{ CIOPT(I(n)^{rand})}=1  \ \ \text{ a.s.} \label{eq:asymp_opt}
\end{align}
\end{conjecture}


It remains unclear whether \textit{Reopt} is asymptotically optimal a.s. for rates $\lambda\geq\tilde{\lambda}$. To prove a somewhat weaker notion of convergence, we fix the time interval $[0,t]$ (e.g., a shift) and examine the orders  $I(t)^{rand}$ arrived in this time interval $[0,t]$ given the arrival rate $\lambda$ of the underlying Poisson process. Recall that, by definition, the number of orders $N(t)$ in the interval is Poisson($\lambda\cdot t$)-distributed. We denote by $\mathbb{P}_{\lambda}$ the underlying probability measure for a given arrival rate $\lambda$. 

\begin{theorem}
\label{th:stoch_optimality}
Let $\lambda_i>0, i \in \mathbb{N}$ be any increasing sequence of arrival rates such that $\lambda_i\rightarrow \infty$ if $i\rightarrow \infty$. In \textit{Pcart-N}, \textit{Pcart}, and \textit{Robot},  the following holds true: 
 \begin{align}
&&&     \lim_{i\rightarrow \infty} \ \mathbb{E}_{\lambda_i}[\frac{Reopt(I(t)^{rand})}{CIOPT(I(t)^{rand})}]= 1 \label{eq:expected_optimality}
 \end{align}
 And for all $\epsilon>0$: 
 \begin{align}
&&&\lim_{i\rightarrow \infty} \ \mathbb{P}_{\lambda_i}(\frac{Reopt(I(t)^{rand})}{CIOPT(I(t)^{rand})} \geq 1+ \epsilon ) =0 \label{eq:stochastic_optimality}
\end{align}
\end{theorem}

\begin{proof}The detailed proof is provided in the Electronic Companion~\ref{sec:Proof_stoch_opt}. 

Recall that to establish convergence, Theorem~\ref{th:asymp_opt_makesp} depended on a sublinear growth of the arrival time $R_n$ of the $n^{th}$ order ($\frac{R_n}{n}\xrightarrow[a.s.]{n\rightarrow\infty} 0$). 
This is not the case when the arrival time sequence is a Poisson process 
with some given inter-arrival time $\frac{1}{\lambda}$.  Therefore, we must recur to a very different line of proof. In this proof, we perform a formal transition from the probability space described in Section~\ref{sec:orderstat}  to the probability space of this section. The proof proceeds in three steps.

First, we want to leverage the convergence statement of  Theorem~\ref{th:asymp_opt_makesp}. Therefore, we apply it in the case of instances $I(n)^{rand}$ with independent uniformly$[0,t]$-distributed arrival times for an arbitrary $t\in \mathbb{R}^+$. Since the ratio $\frac{Reopt(I(n)^{rand})}{CIOPT(I(n)^{rand})}$ is bounded (cf. Section~\ref{sec:abs_CR}), we can restate Theorem~\ref{th:asymp_opt_makesp} as the \textit{convergence in mean}.

Secondly, we assume that the order arrival times follow a Poisson process with some given arrival rate $\lambda\in \mathbb{R}^+$. We examine instances $I(t)^{rand}$  comprising all orders arriving within a predefined time interval $[0,t], t\in \mathbb{R}^+$. Let $N(t)$ be the random number of these orders. 
We then use the  Order Statistics Property of the homogeneous Poisson process, in combination with the statements of the first step of this proof, to derive the \textit{conditional} convergence in mean:
\begin{align}
&&&\lim_{n\rightarrow \infty }\mathbb{E}_{\lambda}[\frac{Reopt(I(t)^{rand})}{CIOPT(I(t)^{rand})} \ \vert \ N(t) =n]=1 \qquad 
\end{align} Convergence implies that for any $\epsilon>0$, we can find $n_{\epsilon}\in \mathbb{N}$ such that $\mathbb{E}_{\lambda}[\frac{Reopt(I(t)^{rand})}{CIOPT(I(t)^{rand})} \ \vert \ N(t) =n]\leq 1+\frac{\epsilon}{2}$ for all $n>n_{\epsilon}$. 

Thirdly, we compute the \textit{unconditional} expectation $\mathbb{E}_{\lambda}[\frac{Reopt(I(t)^{rand})}{CIOPT(I(t)^{rand})}]$ using the Law of Total Expectation for distinct groups of instances --- those with $N(t)> n_{\epsilon}$ and those with $N(t)\leq n_{\epsilon}$. After some transformations, we show that for any sequence of rates $\lambda_i$ with $\lambda_i \xrightarrow{i\rightarrow \infty}\infty$ and any $\epsilon>0$, one can choose $i_{\epsilon}\in \mathbb{N}$ in dependence of $t$ and $\epsilon$, such that $\mathbb{E}_{\lambda_{i}}[\frac{Reopt(I(t)^{rand})}{CIOPT(I(t)^{rand})}]\leq (1+\epsilon)$ for all $i\geq i_{\epsilon}$.

We use the Markov's inequality to derive the implication of (\ref{eq:stochastic_optimality}) from (\ref{eq:expected_optimality}). 
\end{proof}
\subsection{Performance of \textit{Reopt} in general large OOBSRP instances}\label{sec:asy_comp_ratio}
In this section, we abolish the stochastic assumptions and examine whether \textit{Reopt} retains its good performance for general large problem instances. We show that even in the \textit{worst-case}, the \textit{Reopt} result cannot be improved by more than 50\% in any given OOBSRP instance if it is sufficiently large. This is still an excellent performance.

Overall, having examined different worst-case examples for \textit{Reopt} in OOBSRP with a pushcart and a robot, we see two underlying mechanisms of the \textit{Reopt'}s suboptimal result: \textit{(i)} unfortunate batching and, in case of the interventionist \textit{Reopt}, \textit{(ii)} unnecessary walking. In \textit{(i)}, \textit{Reopt} places orders in one batch even if the savings from their batching are small compared to a separate picking of these orders (cf. Example~\ref{ex:example1}). $CIOPT$, to the contrary, can anticipate upcoming better matching orders. In \textit{(ii)}, \textit{Reopt} completely changes the current route of the picker, including her next destination, at the slightest promise of time-saving. Paradoxically, this behavior may result in a prolonged back-and-forth walking without picking an item, which is avoided in $CIOPT$ (cf. the proof of Proposition~\ref{prop:asymp_cr_LBs}).

In the following, Proposition~\ref{prop:makespanub_bin} establishes the upper bound of 2 for  $\sigma_{asymp.}(Reopt)$ in all the examined picking systems. Proposition~\ref{prop:asymp_cr_LBs} states that this worst-case ratio is \textit{tight} for \textit{Pcart} and \textit{Robot}.

\begin{proposition}\label{prop:makespanub_bin}
In \textit{Pcart-N}, \textit{Pcart}, and \textit{Robot}, $\sigma_{\text{asymp.}}(Reopt)\leq2$. In other words, for every instance $I(n)$
\begin{align}
Reopt(I(n))\leq 2 \cdot  CIOPT(I(n))+ const  
\end{align}
where $const$ is the warehouse traversal time $u$ from Lemma~\ref{prop:cover}.
\end{proposition}
\begin{proof} 
By Lemma~\ref{prop:UB1} and Lemma~\ref{prop:simple}:
\begin{align}
    Reopt(I(n)) &\leq r_n + u +  CIOPT(I(n)^{r=0}) \nonumber \\
    & \leq 2 \cdot  CIOPT(I(n)) + u  \label{eq:makespanub_bin}
\end{align}

\end{proof}

\begin{proposition} \label{prop:asymp_cr_LBs}
In $Pcart$ and $Robot$, the upper bound of 2 is tight for the asymptotic competitive ratio, i.e.
\begin{align}
    \sigma_{asymp.}(Reopt)\geq 2
\end{align}  
Note that in the case of $Pcart$, the statement requires a warehouse with at least two aisles. 
\end{proposition}

\begin{figure}[htbp]
    \centering
    \includegraphics[scale=0.7]{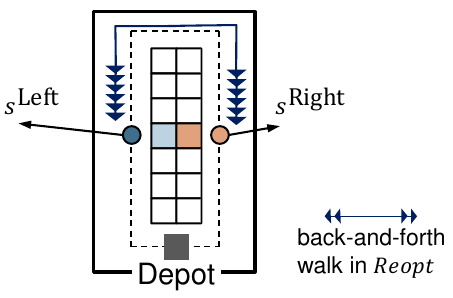}
    \caption{\textit{Pcart}: Unfortunate large instance for \textit{Reopt} \label{fig:Asy_ex} \\ \scriptsize{ Note. $c=2$. Single-item orders with items $s^\text{left}$ and $s^\text{right}$ arrive such that, in \textit{Reopt}, the picker keeps walking back and forth without completing any batch until the arrival of the last order.}}
\end{figure}

\begin{proof}
We examine the case of \textit{Pcart} and refer to Electronic Companion~\ref{sec:proof_asyp_cr_LBs} for the case of \textit{Robot}.

Figure~\ref{fig:Asy_ex} provides an unfortunate example $I^\text{worst}$ for $c=2$ and a two-aisle warehouse. This example can be generalized for $c>2$ and more general warehouse layouts. 

Arriving orders in $I^\text{worst}$ consist of one item, which is placed either in the left aisle (at position $s^\text{left}$) or in the right aisle (at position $s^\text{right}$). We simplify the notation for these single-order items and write $o_i=s^\text{left}$ or  $o_i=s^\text{right}$, $\forall i \in \mathbb{N}$, respectively.  Recall that $W$ and $L$ are the width and the length of the warehouse, respectively. In this example, the picker can traverse both aisles in a cyclic tour of length $T=2W+2L$ from the depot. Let  $d(l_d,s^\text{left})=(l_d,s^\text{right})=\frac{1}{2}\cdot d(s^\text{left}, s^\text{right})=\frac{1}{2}T$. 

Instance $I^\text{worst}$ consists of $n=2k$ orders, where $k\in \mathbb{N}_{\geq 4}$ is some \textit{uneven} number. Each time, orders $s^\text{left}$ and  $s^\text{right}$ arrive simultaneously. The idea is to set the inter-arrival times of the orders such that \textit{Reopt} makes the picker oscillate between the positions of $s^\text{left}$ and $s^\text{right}$ in the upper half of the warehouse,  without picking any item, until the last order arrives.

At time $r_1=r_2=0$, \textit{Reopt} batches the only two available orders together: $S_1=\{o_1,o_2\}$ with $o_1=s^\text{left}$ and $o_2=s^\text{right}$, and plans to pick them in a circular tour. \textit{W.l.o.g.}, let \textit{Reopt} pick $o_1$ first (at time $\frac{1}{4}T$). 

At time $r_3=r_4=\frac{3}{4}T-\frac{\delta}{2k-2}$, an order $s^\text{left}$ and an order $s^\text{right}$ arrive. In \textit{Reopt}, the picker is currently at distance  $\frac{\delta}{2k-2}$ above item $s^\text{right}$ with $o_1$ placed in her cart. \textit{Reopt} has two choices. The first is to follow the initial plan and pick batch $\{o_1,o_2\}=\{s^\text{left},s^\text{right}\}$ to the end, then pick batch $\{o_3,o_4\}=\{s^\text{left},s^\text{right}\}$, at a total cost of $2T$. The second is to turn back and complete a modified batch $S_1=\{o_1,o_3\}=\{s^\text{left},s^\text{left}\}$, then pick $S_2=\{o_2,o_4\}=\{s^\text{right},s^\text{right}\}$, at a total cost of $r_4+ \frac{3}{4}T- \frac{\delta}{2k-2}+\frac{T}{2}=2T-2 \cdot \frac{\delta}{2k-2}$. \textit{Reopt} decides for the second option.

At time $r_5=r_6=\frac{5}{4}T-3 \cdot \frac{\delta}{2k-2}$, an order  $s^\text{left}$ and an order  $s^\text{right}$ arrive, and the picker is at a distance of $\frac{\delta}{2k-2}$ above $s^\text{left}$ with $o_1$ placed in her cart. \textit{Reopt} has two choices. The first is to continue picking $S_1=\{o_1,o_3\}=\{s^\text{left},s^\text{left}\}$ and then $S_2=\{o_2,o_4\}=\{s^\text{right},s^\text{right}\}$ and $S_3=\{o_5,o_6\}=\{s^\text{left},s^\text{right}\}$, at a total cost of $2T-2 \cdot \frac{\delta}{2k-2}+T=3T-2\cdot \frac{\delta}{2k-2}$. The second is to change the first batch again to $S_1=\{o_1,o_2\}=\{s^\text{left},s^\text{right}\}$, then pick $S_2=\{o_3,o_5\}=\{s^\text{left},s^\text{left}\}$ and $S_3=\{o_4,o_6\}=\{s^\text{right},s^\text{right}\}$, at a total cost of $r_6+\frac{3}{4}T-\frac{\delta}{2k-2}+2\cdot \frac{T}{2}=3T-4 \cdot \frac{\delta}{2k-2}$. \textit{Reopt} decides for the second option.

In general, at time $r_{2p-1}=r_{2p}=\frac{2p-1}{4}T-(2p-3) \cdot \frac{\delta}{2k-2}$ for $p\in \{2,..,k\}$, an order $s^\text{left}$ and an order  $s^\text{right}$ arrive. If $p$ is uneven, the picker is at distance $\frac{\delta}{2k-2}$ above $s^\text{left}$ planning to pick $\frac{p-1}{2}$ batches of type $\{s^\text{left},s^\text{left}\}$ first, then $\frac{p-1}{2}$ batches of type $\{s^\text{right},s^\text{right}\}$. With the two recently arrived orders, \textit{Reopt} decides to pick batch $S_1=\{o_1,o_2\}=\{s^\text{left},s^\text{right}\}$ instead, then $\frac{p-1}{2}$ batches of type $\{s^\text{left},s^\text{left}\}$ and $\frac{p-1}{2}$ batches of type $\{s^\text{right},s^\text{right}\}$. So, the picker turns around towards $s^\text{right}$. Similarly, if $p$ is even, the picker is at distance $\frac{\delta}{2k-2}$ above $s^\text{right}$ planning to pick  $S_1=\{o_1,o_2\}=\{s^\text{left},s^\text{right}\}$ first, then $\frac{p}{2}-1$ batches of type $\{s^\text{left},s^\text{left}\}$, and then $\frac{p}{2}-1$ batches of type $\{s^\text{right},s^\text{right}\}$. With the two recently arrived orders, \textit{Reopt} decides to pick  $\frac{p}{2}$ batches of type $\{s^\text{left},s^\text{left}\}$ first, then $\frac{p}{2}$ batches of type  $\{s^\text{right},s^\text{right}\}$. So, the picker turns around towards $s^\text{left}$.  

If $k$ is uneven, $Reopt(I^\text{worst})=kT-\delta$, which equals the release time of the last order $r_{2k}=\frac{2k-1}{4}T-(2k-3) \cdot \frac{\delta}{2k-2}$ plus the time $(\frac{3}{4}T-\frac{\delta}{2k-2})$ to reach $s^\text{right}$ and complete the first batch plus $(\frac{k-1}{2}\cdot 2 \cdot \frac{1}{2}T)$ to pick $\frac{k-1}{2}$ batches of type $\{s^\text{left},s^\text{left}\}$ and $\frac{k-1}{2}$ batches of type $\{s^\text{right},s^\text{right}\}$. In contrast, $CIOPT$ picks $S_1=\{o_1,o_2\}$, then $\frac{k-1}{2}$ batches of type $\{s^\text{left},s^\text{left}\}$ and $\frac{k-1}{2}$ batches of type $\{s^\text{right},s^\text{right}\}$ such that $CIOPT(I(k))=T+(k-1)\frac{T}{2}=\frac{k}{2}T+\frac{T}{2}$. This implies that 
\begin{align}
    \lim_{k \rightarrow \infty} \quad \frac{Reopt(I(k))}{CIOPT(I(k))}=2 \label{eq:limit_asymp_cr}
\end{align}
By Lemma~\ref{lem:lim_asymp_cr} in Electronic Companion~\ref{sec:aux_lemma_asymp_cr},  $\sigma_{asymp}(Reopt)\geq 2$ follows immediately for \textit{Pcart}. 

For $c>2$, we can place multiple single-item orders at positions $s^\text{left}$ and $s^\text{right}$. Note that in the example above, \textit{Reopt} faces ties. Initially, there are two alternative plans with the same cost -- picking the orders together as $S_1=\{o_1, o_2\}$ in a cyclic tour or picking them separately as $S_1=\{o_1\}$ and $S_2=\{o_2\}$. In the second plan, the example does not work. We can avoid ties by considering $o_1=\{s^\text{left},s'\}$, where $s'$ is placed just above $s^\text{left}$, which makes \textit{Reopt} to prefer picking both orders together.  
\end{proof}

\subsection{Performance of \textit{Reopt} in  instances of all sizes} \label{sec:abs_CR}

In \textit{any} OOBSRP instance, \textit{Reopt} never exceeds more than 4 times the complete-information optimum in \textit{Pcart} and \textit{Robot} and no more than 2.5 times in \textit{Pcart-N}, respectively (see Proposition~\ref{prop:strict_CR}). However, strict competitive ratios of \textit{any} deterministic online algorithm are at least 2, 1.64, and 2 in \textit{Pcart-N}, \textit{Pcart}, and \textit{Robot}, respectively, by Proposition~\ref{prop:LB_abs_CR_general}. In other words, the performance of \textit{Reopt} cannot be improved much across instances of all sizes.

Observe that the upper bound 2.5 on the strict competitive ratio for \textit{Reopt} in \textit{Pcart-N} is almost tight, as we can construct instances $I$ with $\frac{Reopt(I)}{CIOPT(I)} =2.5-\tilde{\epsilon}$ for an arbitrary small $\tilde{\epsilon}>0$ (see Example~\ref{ex:example1}).

\begin{proposition}\label{prop:strict_CR}
The following upper bounds hold for the strict competitive ratio:
\begin{align}
\text{In \textit{Pcart-N}:}  \qquad    \sigma(Reopt) & \leq 2.5 \label{eq:str_cr_PcartN} \\
\text{In \textit{Pcart} and \textit{Robot}:} \qquad     \sigma(Reopt) & \leq 4 \label{eq:str_cr_Pcart}
\end{align}

\end{proposition} 

\begin{proof}[Proof of (\ref{eq:str_cr_PcartN})]
We are in system \textit{Pcart-N} and  consider two cases.
\end{proof}

In the first case, the picker in \textit{Reopt} is idle at the depot at time $r_n$. Then, by Lemma~\ref{prop:simple}:
\begin{align}
Reopt(I)\leq r_n+ CIOPT(I^{r=0}) \leq 2 \cdot  CIOPT(I) \label{eq:UB_abs_CR_case1}
\end{align}

In the other case, at $r_n$, the picker in \textit{Reopt} is collecting batch $\tilde{S}$, for which she left the depot at time $t(\tilde{S})$.  Let $\tilde{j}$ be the index of the first order that arrives \textit{after} time $t(\tilde{S})$. Let $I_{\geq \tilde{j}}$ denote an instance which contains the subset of orders $\{o_i \vert \ s.t. \ i  \geq \tilde{j}\}$ and assumes these orders are \textit{instantly available}. We denote by $\tilde{J}$ an instance that contains all orders in the queue of \textit{Reopt} at time $t(\tilde{S})$, and that assumes all these orders are \textit{instantly available}. By definition, at time $t(\tilde{S})$, \textit{Reopt}'s plan to collect the orders of $\tilde{J}$ is \textit{optimal} given the current information. Therefore the solution of  \textit{Reopt} is not worse than the following feasible policy: At $t(\tilde{S})$ complete first $\tilde{J}$ in time $CIOPT(\tilde{J})$. After that, the picker is at the depot and collects the remaining orders in at most $CIOPT(I_{\geq \tilde{j}})$ time. We denote $d_{max}=\max\{d(l_d,s_j^i) \ \vert \ s_j^i \in o_{\tilde{j}} \cup o_{\tilde{j}+1} \cup ... \cup o_n \}$, the maximal distance from the depot to an item in $I_{\geq \tilde{j}}$. Altogether:
\begin{align}
Reopt(I)  & \leq t(\tilde{S})+ CIOPT(\tilde{J})+ CIOPT(I_{\geq \tilde{j}})\label{eq:UB_nonint_first}\\
&  \leq  r_{\tilde{j}} +  CIOPT(I_{\geq \tilde{j}})+  CIOPT(I) \qquad \label{eq:UB_nonint_2}\\
& \leq CIOPT(I) + d_{\max} +  CIOPT(I)  \label{eq:UB_nonint}\\
& \leq 2.5 \cdot CIOPT(I)  \label{eq:UB_nonint_final}
\end{align}
Line (\ref{eq:UB_nonint_2}) follows by Lemma~\ref{prop:simple}. In $CIOPT$, the picker can move towards the first picking location prior to the arrival of the corresponding order, thus (\ref{eq:UB_nonint}) ensues from (\ref{eq:UB_nonint_2}): $CIOPT(I)+d_{\max}\geq r_{\tilde{j}}+CIOPT(I_{\geq \tilde{j}})$.  Finally, since $CIOPT(I)\geq 2\cdot d_{\max}$, (\ref{eq:UB_nonint_final}) follows.

The logic above does not hold for \textit{Pcart-Int} and \textit{Robot}, because the initiated batch $\tilde{S}$ can be changed at each order arrival after $t(\tilde{S})$, possibly leading to a worse solution.\\

\begin{proof}[Proof of (\ref{eq:str_cr_Pcart})]
Consider \textit{Pcart} and recall Lemma~\ref{prop:UB1}: At time $r_n$, when the last order arrives, \textit{Reopt} reoptimizes having the complete information and requires not more time than the following simple policy: 1) Complete the currently open batch and move to the depot, then 2) use $CIOPT(I(n)^{r=0})$ to collect the remaining orders. For 1), 
the picker requires at most $r_n+CIOPT(I(n)^{r=0})$, because she can walk from the current position all the way back to the depot in $r_n$ and then follow $CIOPT$ for picking the remaining items of the currently open batch (and skipping the rest). Thus, for 1) together with 2), by Lemma~\ref{prop:simple}:
\begin{align}
Reopt(I(n))\leq & \ 2 \cdot r_n+ 2\cdot CIOPT(I(n)^{r=0}) \nonumber \\
\leq & \ 4\cdot CIOPT(I(n))  \label{eq:UB_reopt}
\end{align}
The proof for \textit{Robot} proceeds along the same lines. Observe, however, that at the start of step 2) the picker may be away from the depot. However, she can follow $CIOPT$ in the  \textit{reverse} direction to collect the remaining items. 
\end{proof}

\begin{figure}[htbp]
    \centering
    \includegraphics[scale=0.4]{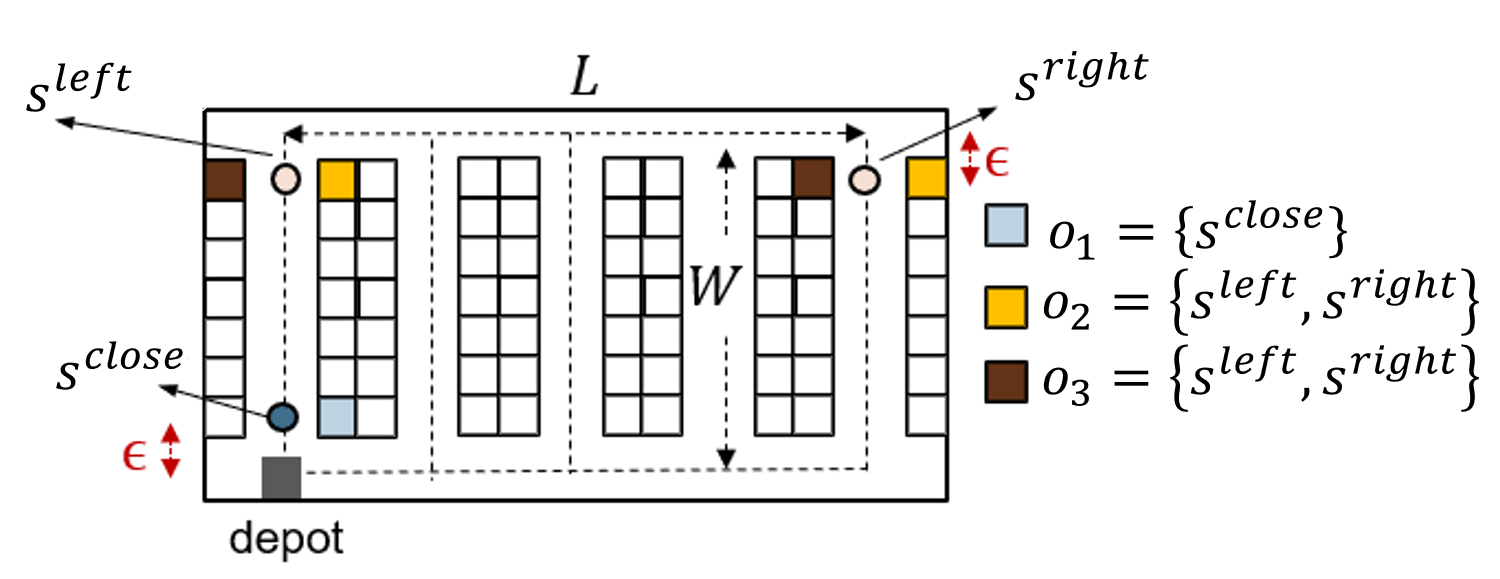}
    \caption{\textit{Pcart-N}: Unfortunate small instance $I^{unf}$ for \textit{Reopt}. \label{fig:LB_abs_CR_Reopt} \\ \scriptsize{ Note. $c=2, r_1=r_2=L+W, \text{and } r_3=L+W+\delta$ with $\delta\rightarrow 0$. \textit{Reopt} myopically collects poorly matching orders $o_1$ and $o_2$ together.}}
\end{figure}

Figure~\ref{fig:LB_abs_CR_Reopt} provides an unfortunate instance $I^{unf}$ for  \textit{Pcart-N} with $c=2$, such that $\frac{Reopt(I^{unf})}{CIOPT(I^{unf})}\approx 2.5$. 

\begin{example}\label{ex:example1}
In instance $I^{unf}$ of Figure~\ref{fig:LB_abs_CR_Reopt}, \textit{Reopt} receives two orders at time $r_1=r_2=W+L$: $o_1$ with only one item $s^{close}$ at the closest picking location to the depot, $d(l_s,s^{close})=\epsilon$, and $o_2$ with two far away picking locations that can be visited best in a cyclic tour of length $2(W+L)$ from the depot. \textit{Reopt} forms batch $s_1=\{o_1,o_2\}$, and as soon as the picker leaves the depot (at time $r_3=W+L+\delta$, $\delta\rightarrow 0$), a third order $o_3=o_2$ arrives. Although $o_2$ and $o_3$ are perfectly matching, \textit{Reopt} in $Pcart-N$ cannot modify the commenced batch $S_1$, and completes the operation with $S_2=\{o_3\}$ at time $Reopt(I^{unf})=r_1+2(W+L)+2(W+L)=5(W+L)$. On the other hand, $CIOPT$ picks the first item of $o_2$ and $o_3$ (in $s^{right}$) at time $W+L+\delta$ and completes its first batch $\{o_2,o_3\}$ at time $2(L+W)+\epsilon+\delta$. After picking $o_1$ separately, $CIOPT(I^{unf})=2(L+W)+3\epsilon+\delta$. Since $\delta\rightarrow 0$ and $\epsilon$ is the distance from the depot to the closest picking location, we can construct instances with arbitrarily small $(3\epsilon+\delta)>0$.  
\end{example}

We now denote by $ALG$ an arbitrary deterministic online algorithm for OOBSRP in each of the studied systems (\textit{Pcart-N}, \textit{Pcart}, and \textit{Robot}). An algorithm is \textit{deterministic}, if
for every fixed instance, it always generates the same result when run repeatedly.
We differentiate deterministic algorithms from probabilistic algorithms, which may choose their actions randomly. 
Note that $ALG$ in \textit{Pcart-N} must be an arbitrary \textit{non-interventionist} algorithm, that can update the current picking plan only when the picker is at the depot.
\begin{proposition}\label{prop:LB_abs_CR_general}
For every deterministic online algorithm $ALG$ for $OOBSRP$, the strict competitive ratio cannot be better as the following:

\begin{align}
\text{In \textit{Pcart-N} and \textit{Robot}:}  \qquad    \sigma(ALG) & \geq 2 \label{eq:general_LB_nonint} \\
\text{In \textit{Pcart} :} \qquad     \sigma(ALG) & \geq 1.64 \label{eq:general_LB_Pcart_int}
\end{align}
Note that we require $c\geq 2$ for \textit{Pcart-N} and \textit{Pcart}.
\end{proposition}

\begin{figure}[htbp]
    \centering
    \includegraphics[scale=0.7]{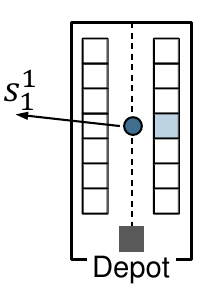}
    \caption{\textit{Pcart-N}: Unfortunate small instance for all algorithms. \label{fig:Worst_all} \\ \scriptsize{ Note. No action suits all possible future well: a) no further orders arrive, b) further orders arrive.}}
\end{figure}

\begin{proof}
The lower bounds for \textit{Pcart} and \textit{Robot} are based on Theorem 3.1  and Theorem 3.3 of \citet{Ausiello1999}, respectively. In both cases, \citet{Ausiello1999} construct unfortunate instances for the online TSP \textit{on the real line}. These instances can be transformed to the warehouse environment. The proof requires nontrivial extensions, which are given in the Electronic Companion~\ref{sec:proof_LB_abs_general}. Indeed, in the warehouse, the movements of the picker are not limited to a real line and, moreover, the batching capacity should be respected. 

To prove the lower bound for any $ALG$ in \textit{Pcart-N}, consider the instance in Figure~\ref{fig:Worst_all}. Recall that the warehouse width is $W$ and the picker moves at the unit speed. Initially, at $r_1=0$, order $o_1=\{s_1^1\}$ with $d(l_d,s_1^1)=\frac{1}{2}W$ is available. Let $ALG$ decide to dispatch the picker to collect the single-order batch $\{o_1\}$ at time $t$. In the following, we will construct an unfortunate instance for each possible value of $t$. 

If $t\geq W$, no more orders arrive.  $ALG$ needs $(t+W)\geq2W$ time to collect $o_1$ in a return-trip from the depot, while $CIOPT(I)= W$ for the resulting instance $I$. 

 If $\frac{1}{2}W\leq t < W$, then a second order $o_2=\{s_2^1\}$ with $d(l_d,s_2^1)=t$ arrives at time $r_2=t$ (but after the picker dispatched). Then, in the resulting instance $I$, batch $\{o_1,o_2\}$ can be collected in time $CIOPT(I)=2t$, while  $ALG$ collects $o_1$ and $o_2$ in two separate batches with $ALG^{Pcart-N}(I)=t+W+2t>4t$.

 If $ t < \frac{1}{2}W$, then a second order $o_2=\{s_2^1\}$ with $d(l_d,s_2^1)=\frac{1}{2}W$ arrives at time $r_2=t$ (but after the picker dispatched). For the resulting instance $I$, it is possible to collect both orders in the single batch $\{o_1,o_2\}$ in time $CIOPT(I)=W$, while $ALG(I)=t+W+W\geq2W$, because it picks $o_1$ and $o_2$ in  separate batches.

In all the cases above, the ratio $\frac{ALG}{CIOPT(I)}\geq2$ for the respective instance $I$ in \textit{Pcart-N}.

Observe that in the unfortunate examples for \textit{Pcart-N} and \textit{Pcart}, $c\geq2$. If $c=1$, a weaker lower bound of 1.5 can be straightforwardly derived in both cases from the unfortunate instance created for OOBSRP in \textit{Robot} in the Electronic Companion~\ref{sec:proof_LB_abs_general}. 
\end{proof}

\section{Computational experiments}\label{sec:experiments}
The analysis of Section~\ref{sec:asymp_opt} showed that \textit{Reopt} may converge to an optimal policy with an increasing instance size. In this section, we compare the results of \textit{Reopt} to optimality (measured by $CIOPT$) experimentally and investigate the gap to optimality in small instances.

Note the computational challenges associated with these experiments. In order to make representative measurements across diverse warehouse settings, we create 2600 OBSRP$^*$ instances and solve more than 2380 of those to \textit{optimality}. The quantity of items that need to be picked in the created instances ranges between 4 and 50, and OBSRP$^*$, which combines batching and routing, is NP-hard in the strong sense \citep[c.f. ][]{Loeffler2022}. 
Given the scale of the endeavor, we limit our studies to the interventionist versions of \textit{Reopt}, i.e., systems \textit{Pcart} and \textit{Robot}.
Section~\ref{sec:exp_framework} explains the data generation and Section~\ref{sec:exp} reports on experiments.  
\subsection{Data sets} \label{sec:exp_framework}

We randomly generate instances grouped in 10 settings (see Table~\ref{tab:OverSettings}) that feature different warehouse environments. Each setting comprises 10 instances of each order size $n\in\{3, 4, \ldots, 15\}$. The procedure is repeated for two picking systems -- \textit{Pcart} and \textit{Robot}. Thus in total, 2600 instances (2 systems $\times$ 10 settings $\times$ 13 sizes $\times$  10 ) are created. 

In \textit{the Base setting }\textit{(Base)}, the picker's speed equals 0.8~m/sec \citep[cf. ][]{Henn2012} and it takes 10~sec to collect one item at each picking location. The warehouse consists of 10 aisles and 3 cross-aisles, each of which has a width of 3~m. One single shelf row marks the left and right outer borders of the storage area, respectively, and neighboring aisles are separated by two shelf rows placed alongside each other (cf. Figure~\ref{fig:warehouse}). The shelves measure 3~m in width and 1~m in depth. One aisle contains 20 shelves. Thus, the warehouse has a length of 50~m (10 aisles $\times$ 3~m $+$ $(1+9\cdot2+1)$ shelf rows $\times$ 1~m) and a width of 69~m (20 shelves $\times$ 3~m $+$ 3 cross-aisles $\times$ 3~m). Each shelf has slots for three distinct \textit{stock keeping units (SKUs)}, i.e., distinguishable product items, which results in a granularity of 1~m and provides $1200$ slots for SKUs (20 shelves per aisle $\times$ 3 $\times$ 20 shelf rows)   in the warehouse in total. The depot is positioned at the lower-left corner of the warehouse. The batch capacity is $c=2$. The number of items $k(j)$ in each order $o_j$ is uniformly distributed in set the $ \{1,2,...,\overline{k}\}, \overline{k}=4$, which reflects a common scenario in e-commerce, where the typical order contains fewer than two items on average \citep[c.f.][]{Boysen2019, Xie2023}. 
The storage locations of ordered items are uniformly distributed among all storage locations \textit{(uniform dispersion)}. The orders arrive according to a homogeneous Poisson process with an arrival rate $r$.   
In 
preliminary experiments, we investigated different arrival rates $r$ for shifts of 4h, increasing the rates in steps of 10 orders/shift, and found that optimal arrival rates in the \textit{Base} setting, at which the system remains stable, are  $r^{\text{Pcart}}=90$ orders/shift and $r^{\text{Robot}}=110$ orders/shift for \textit{Pcart} and \textit{Robot}, respectively. So, we set the arrival rates to these numbers. For \textit{Pcart}, the 130 instances in \textit{Base}, which form 13 groups of 10 instances according to their size $n$, are generated randomly and independently from each other. For better comparability, the 130 instances in \textit{Base} for \textit{Robot} are identical to those in \textit{Pcart}, except for the sequence of arrival times, which was produced using a higher stable rate $r^{\text{Robot}}$.  
We generate the remaining nine settings, by taking the instances of \textit{Base} and changing their specific characteristics in a one-factor-at-a-time design (see Table~\ref{tab:OverSettings}). For better control of the impact of warehouse characteristics, we keep the instances unchanged (including the coordinates of the generated items and the order arrival times) over the different settings, unless explicitly stated otherwise. 

We investigate different batch capacities in settings \textit{Smallbatches} and \textit{Largebatches}, where we set $c=1$ and $c=4$, respectively (all other instance data remained unchanged). In \textit{Smallrate} and \textit{Largerate}, the arrival rates are set to $r^{\text{Pcart}}=70$,  $r^{\text{Pcart}}=110$ and $r^{\text{Robot}}=90$ and $r^{\text{Robot}}=130$ for \textit{Robot} and \textit{Pcart}, respectively. We do it by generating 
new arrival times for the instances in \textit{Base}. We investigate warehouse geometries in \textit{Largewarehouse}  (20 aisles) and \textit{Lesscrossaisles} (2 cross-aisles). Observe that the lengths and widths of the shelves, aisles, and cross-aisles, as well as the number of shelves per aisle, remain the same so that the size of the warehouse changes accordingly. In particular, in \textit{Largewarehouse}, the warehouse has a length of 100~m (20 aisles $\times$ 3~m $+$ $(1+19\cdot2+1)$ shelf rows $\times$ 1~m) and a width of 69~m is identical to \textit{Base}; in \textit{Lesscrossaisles}, the warehouse length of 50~m is identical to \textit{Base}, and the width is 66~m (20 shelves $\times$ 3~m $+$ 2 cross-aisles $\times$ 3~m).  In \textit{Largewarehouse}, we generate the positions of the ordered items anew. We do not have to do it in \textit{Lesscrossaisles}, since the number of SKU slots remains the same (1200) and only the middle cross-aisle disappears. In \textit{Classbaseddispertion}, SKUs are stored in the warehouse according to their turnover rates, so that we take instances of \textit{Base} and generate new picking positions. We randomly assign ordered items to three classes: A, B, and C with probabilities of 52\%, 36\%, and 12\%. Items in A refer to high-turnover products, which occupy slots in the first aisle. Items in B occupy aisles 2 to 4. And aisles in C are low-turnover products, which occupy aisles 5 to 10. We randomly assign a storage position to each ordered item in the respective aisles of its class. Finally, in \textit{Smallorders} and \textit{Largeorders}, the maximal order size is set to $\overline{k}=2$ and $\overline{k}=8$, respectively. For this, we have to generate the orders (including their size, positions of the items, and arrival times) anew in the instances of \textit{Base}. The generated instances can be found in the Electronic Companion.

We performed the experiments on the Compute Canada cluster by using 1 CPU with a time and memory limit of at most 1h and 350G, respectively.  For each instance $I$, we compute $CIOPT(I)$ by solving the respective OBSRP* exactly with the dynamic programming approach described in \citet{Lorenz2024}. We calculate $Reopt(I)$ by letting the orders arrive dynamically according to their arrival rates and applying exact reoptimization as described in Section~\ref{sec:onlineproblem} using the straightforwardly adapted dynamic programming approach of \citet{Lorenz2024}.

\begin{table}
\caption{Overview of the 10 settings of the data set }\label{tab:OverSettings}
\centering
\scriptsize{
\begin{tabular}{l p{0.2\linewidth}p{0.2\linewidth}p{0.2\linewidth}}
\toprule
Changed factor &\multicolumn{3}{c}{Settings}\\
\midrule
Number of aisles&  --  & Base & Largewarehouse   \\
&    & (10 aisles) & (20 aisles)   \\[0.2cm]

Number of cross-aisles& Lesscrossaisles   & Base &  --  \\
&  (2 cross-aisles)  & (3 cross-aisles) &   \\[0.2cm]

Storage policy&  -- & Base &  Classbaseddispersion  \\
&&  (uniform)  & (class-based) \\[0.2cm]

Maximal order size& Smallorders   & Base &  Largeorders  \\
&  (2 SKUs)  & (4 SKUs) &  (8 SKUs) \\[0.2cm]

Batching capacity& Smallbatches   & Base &  Largebatches  \\
&  ($c=1$)  & ($c=2$) &  ($c=4$) \\[0.2cm]

Arrival rate& Smallrate   & Base &  Largerate  \\
&  ($r^{\text{Pcart}}=70$,   & ($r^{\text{Pcart}}=90$,  &  ($r^{\text{Pcart}}=110$, \\
&  $r^{\text{Robot}}=90$)  & $r^{\text{Robot}}=110$) &  $r^{\text{Robot}}=130$) \\

\midrule[\heavyrulewidth]
  \end{tabular}}
\end{table}

\subsection{Optimality ratios of \textit{Reopt} in \textit{Pcart} and \textit{Robot}} \label{sec:exp}
For each instance $I$, we compute the \textit{optimality ratio} as $\frac{Reopt(I)}{CIOPT(I)}$. 
Tables~\ref{tab:compratio_worst_Pcart} and \ref{tab:compratio_worst} report worst observed optimality ratios in \textit{Pcart} and \textit{Robot}. The average observed optimality ratios for \textit{Pcart} and \textit{Robot} can be found in Tables~\ref{tab:compratio_avg_Pcart} and \ref{tab:compratio_avg} in Electronic Companion~\ref{sec:exp_ratios_Pcart}.  

According to our experiments, \textit{Reopt}'s results are rather close to optimality even in small-sized instances. For the smallest tested instance sizes (with the number of orders $n<5$), the average ratio over all settings equaled 1.10 and 1.09, for \textit{Pcart} and \textit{Robot}, respectively; for the largest tested instance sizes ($n>13$), it decreased to 1.05, for both cart types. The worst observed optimality ratio across \textit{all} the instances equaled 1.34 for \textit{Pcart} and 1.32 for \textit{Robot}; and the respective ratios for larger instances ($n>13$) decreased and equaled 1.13 and 1.16, respectively.  

To statistically validate the monotonic decreasing relationship between the observed optimality ratios and the instance size 
$n$, we conducted a Spearman's rank correlation test \citep[see e.g.,][]{Corder2009}. For each of the 10 settings, we computed the Spearman's correlation coefficient for the optimality ratios across all solved instances of different instance sizes (130 observations for most settings, fewer for Largeorders and Largebatches, see Tables~\ref{tab:compratio_worst_Pcart} and~\ref{tab:compratio_worst}). In \textit{Pcart} (\textit{Robot}), for all settings except Largeorders, the coefficients ranged from -0.68 to -0.25  (from -0.61 to -0.27) with p-values of less than 0.04 (less than 0.01) indicating a moderate but statistically significant decrease. For Largeorders, the Spearman's correlation coefficient was -0.23 (-0.16) and the p-value of 0.07 (0.20) may be explained by the lower number of observations and the limitation in instance size $n$.

\begin{table}
\caption{Worst observed optimality ratios of \textit{Reopt} in   \textit{Pcart}}\label{tab:compratio_worst_Pcart} 
\scriptsize{
\begin{tabular}{lrrrrrrrrrrrrr}
\toprule
Setting  & \multicolumn{13}{c}{\underline{The number of orders $n=$}}\\[0.05cm]
 &  3 & 4 & 5 & 6 & 7 & 8 & 9 & 10 & 11 & 12 & 13 & 14 & 15  \\
\midrule
Base & 1.18 & 1.14 & 1.16 & 1.12 & 1.15 & 1.09  & 1.20 & 1.08 & 1.14 & 1.14 & 1.07 & 1.10 & 1.11  \\
Largewarehouse & 1.21 & 1.15 & 1.15 & 1.18 & 1.18 & 1.12 & 1.18 & 1.15 & 1.11 & 1.12 & 1.13 & 1.12 & *1.10 \\
Lesscrossaisles & 1.16 & 1.24 & 1.18 & 1.10 & 1.10 & 1.14 & 1.14 & 1.10 & 1.12 & 1.09 & 1.07 & 1.13 & 1.13  \\
Classbaseddispersion & 1.20  & 1.14 & 1.18 & 1.12 & 1.16 & 1.11  &  1.14 & 1.14 & 1.13 & 1.11 & 1.10 & 1.10 & *1.08  \\
Smallorders & 1.34 & 1.33 & 1.22 & 1.11  & 1.16 & 1.08 & 1.16 & 1.08 & 1.07 & 1.09 & 1.23 & 1.10 & 1.09  \\
Largeorders & 1.18 & 1.21 & 1.10 & 1.08 & 1.12 & *1.17 & *1.10 & -- & --  & -- & -- & -- & --  \\
Smallbatches & 1.15 & 1.15 & 1.08 & 1.08 & 1.07 & 1.06 & 1.06 & 1.05 & 1.05 & 1.04 & 1.04 & 1.04 & 1.03  \\
Largebatches  & 1.17  & 1.17 & 1.31 & 1.16  & 1.17 & 1.12 & 1.13 & -- & -- & -- & -- & -- & --   \\
Smallrate & 1.14 & 1.17 & 1.13 & 1.11 & 1.08 &  1.12 & 1.05 & 1.05 & 1.10 & 1.15 & 1.08 & 1.12 & 1.09   \\
Largerate  & 1.19  & 1.14 & 1.19 & 1.21 & 1.12  & 1.11 & 1.13 & 1.12 & 1.12 & 1.14 & 1.13 & 1.11 & *1.08  \\
\midrule
Total  & 1.34  & 1.33 & 1.31 & 1.21 & 1.18 & 1.17 & 1.20 & 1.15 & 1.14 & 1.15 & 1.23 & 1.13 & 1.13 \\
\midrule[\heavyrulewidth]
\multicolumn{14}{l}{\scriptsize{\textit{Note.} $^*$ Results for less than 10 (out of 10) instances. CIOPT did not solve 1 or 2 instances to optimality.}}\\
\multicolumn{14}{l}{\scriptsize{\quad\quad~~ -- No results reported, since CIOPT solved less than 8 instances to optimality.}}
  \end{tabular}}
\end{table}

\begin{table}
\caption{Worst observed optimality ratios of \textit{Reopt} in  \textit{Robot}}\label{tab:compratio_worst}  
\scriptsize{
\begin{tabular}{lrrrrrrrrrrrrr}
\toprule
Setting & \multicolumn{13}{c}{\underline{The number of orders $n=$}} \\[0.05cm]
  &  3 & 4 & 5 & 6 & 7 & 8 & 9 & 10 & 11 & 12 & 13 & 14 & 15  \\
\midrule
Base & 1.18 & 1.24 & 1.12 & 1.10 & 1.18 & 1.05 & 1.12 & 1.09 & 1.06 & 1.09 & 1.07 & 1.11 & 1.08   \\
Largewarehouse & 1.32 & 1.14 & 1.19 & 1.24 & 1.12 & 1.08 & 1.12 & 1.09 & 1.07 & 1.08 & 1.11 & 1.09 & 1.10  \\
Lesscrossaisles & 1.25 & 1.25 & 1.31 & 1.12 & 1.16 & 1.13 & 1.13 & 1.10 & 1.17 & 1.11 & 1.14 & 1.16 & 1.12 \\
Classbaseddispersion & 1.15  & 1.22 & 1.10 & 1.18 & 1.14 & 1.18 & 1.08 & 1.09 & 1.08 & 1.04 & 1.08 & 1.10 & 1.10  \\
Smallorders & 1.19 & 1.29 & 1.17 & 1.11 & 1.10 & 1.08 & 1.13 & 1.07 & 1.08 & 1.06 & 1.06 & 1.12 & 1.07   \\
Largeorders & 1.12 & 1.11 & 1.22 & 1.07 & 1.10 &  *1.12 & *1.10 & -- & -- & -- & -- & -- & -- \\
Smallbatches & 1.22 & 1.24 & 1.19 & 1.06 & 1.15 & 1.08 & 1.14 & 1.11 & 1.09& 1.07 & 1.10 & 1.06 & 1.06 \\
Largebatches  & 1.18  & 1.24 & 1.15 & 1.09 & 1.11 & 1.11 & 1.15 & *1.12 & -- & -- & -- & -- & --  \\
Smallrate & 1.22 & 1.15 & 1.11 & 1.08 & 1.07 & 1.09 & 1.09& 1.10 & 1.10 & 1.09 & 1.06 & 1.10 & 1.08  \\
Largerate  & 1.25  & 1.12 & 1.23 & 1.08  & 1.17 & 1.09 & 1.15 & 1.10 & 1.09 & 1.09 & 1.13 & 1.11 & *1.08 \\
\midrule
Total  & 1.32  & 1.29 & 1.31 & 1.24 & 1.18 & 1.18 & 1.15 & 1.12 & 1.17 & 1.11 & 1.14 & 1.16 & 1.12 \\
\midrule[\heavyrulewidth]
\multicolumn{14}{l}{\scriptsize{\textit{Note.} $^*$ Results for less than 10 (out of 10) instances. CIOPT did not solve 1 or 2 instances to optimality.}}\\
\multicolumn{14}{l}{\scriptsize{\quad\quad~~ -- No results reported, since CIOPT solved less than 8 instances to optimality.}}
  \end{tabular} }

\end{table}

\section{Conclusion}\label{sec:conclusion}
This paper performs an optimality analysis of reoptimization policies for \textit{the Online Order Batching, Sequencing, and Routing Problem (OOBSRP)}. The focus is on \textit{immediate reoptimization (Reopt)}, where the current (partial) solution adjusts each time a new order arrives. We consider warehouses employing manual pushcarts and warehouses employing robotic carts. Given that picking instructions in traditional warehouses may only be transmitted at the depot, we additionally examine the \textit{noninterventionist} version of \textit{Reopt} for manual pushcarts. In total, we investigate three systems, each featuring a specific reoptimization policy and cart technology: \textit{Robot} for interventionist \textit{Reopt} and robotic carts, \textit{Pcart} for interventionist \textit{Reopt} and manual pushcarts, as well as \textit{Pcart-N} for noninterventionist \textit{Reopt} and manual pushcarts. 

OOBSRP can be interpreted as a nontrivial combination of the \textit{online batching problem} and the \textit{online traveling salesman} problem. To the best of our knowledge, there is currently no clarity on the performance gaps to optimality for \textit{any} OOBSRP policy. \textit{Reopt} is chosen for analysis due to its prominence in warehousing literature and its ability to  leverage recent algorithmic progress for the offline version of OOBSRP.

Our extended competitive analysis demonstrates the \textit{almost sure asymptotically optimal} performance of \textit{Reopt} under very general stochastic assumptions in all three examined systems -- \textit{Robot}, \textit{Pcart}, and \textit{Pcart-N}.  Our stochastic assumptions can describe a variety of order arrival patterns relevant for practice. On the formal side, they encompass arrival times that are modeled as order statistics or as a homogeneous Poisson process with specific arrival rates. 

Even if we drop stochastic assumptions and examine \textit{any} theoretically possible OOBSRP instances,  \textit{Reopt} performs well, with no policy improving its result by more than 50\% in sufficiently large  instances. Anomalies are possible in small OOBSRP instances, which cannot be overcome by \textit{any} deterministic algorithm (see strict competitive ratios  in Section~\ref{sec:abs_CR}). However, no policy can improve \textit{Reopt}'s result by more than 75\% regardless of the instance size. 

Having established the almost sure \textit{asymptotic} optimality of \textit{Reopt} analytically under various stochastic scenarios, we show empirically that \textit{Reopt} remains rather close to the complete-information optimality already in small instances in \textit{Pcart} and \textit{Robot}.

This paper provides insights into ongoing debates regarding the benefits of waiting and anticipation for online policies in warehouse operations. Indeed, the performance gaps of \textit{Reopt},
which eschews both waiting and anticipation, serve as a benchmark for assessing the benefits of these
concepts  (see Section~\ref{sec:contribution}).

Several questions remain for future research, including the open status of  Conjecture~\ref{con:ConjecturePoisson} and determining \textit{stable rates} $\lambda$ when order arrival times follow a Poisson process. Further performance analysis is needed for other online algorithms, including simple policies like  'first-come-first-served' and S-shape routing. An optimal online policy is still unknown, especially for small OOBSRP instances.  Exploration of additional OOBSRP variants and extensions, including multiple pickers, picking location assignment, advanced information on orders, and order due dates is suggested as avenues for future research. Furthermore, the analysis of alternative objectives, particularly those related to the service level (e.g., tardiness), remains an outstanding research area.

\section*{Acknowledgments}\label{sec:acknowledgements}
We would like to extend our gratitude to the Government of Quebec (Ministère des Relations internationales et de la Francophonie) and BayFor for their support. We also acknowledge Compute Canada for providing the computational resources necessary for this research. 

\bibliography{AGV_picking.bib}

\begin{thebibliography}{44}
\providecommand{\natexlab}[1]{#1}
\providecommand{\url}[1]{\texttt{#1}}
\expandafter\ifx\csname urlstyle\endcsname\relax
  \providecommand{\doi}[1]{doi: #1}\else
  \providecommand{\doi}{doi: \begingroup \urlstyle{rm}\Url}\fi

\bibitem[Alipour et~al.(2020)Alipour, Mehrjedrdi, and Mostafaeipour]{Alipour2018}
M.~Alipour, Y.~Z. Mehrjedrdi, and A.~Mostafaeipour.
\newblock A rule-based heuristic algorithm for on-line order batching and scheduling in an order picking warehouse with multiple pickers.
\newblock \emph{RAIRO-Operations Research}, 54\penalty0 (1):\penalty0 101--107, 2020.

\bibitem[Ascheuer et~al.(2000)Ascheuer, Krumke, and Rambau]{Ascheuer2000}
N.~Ascheuer, S.~O. Krumke, and J.~Rambau.
\newblock Online dial-a-ride problems: Minimizing the completion time.
\newblock In \emph{STACS 2000: 17th Annual Symposium on Theoretical Aspects of Computer Science Lille, France, February 17--19, 2000 Proceedings 17}, pages 639--650. Springer, 2000.

\bibitem[Ausiello et~al.(2001)Ausiello, Feuerstein, Leonardi, Stougie, and Talamo]{Ausiello1999}
G.~Ausiello, E.~Feuerstein, S.~Leonardi, L.~Stougie, and M.~Talamo.
\newblock Algorithms for the on-line travelling salesman.
\newblock \emph{Algorithmica}, 29:\penalty0 560--581, 2001.

\bibitem[Bertsimas and van Ryzin(1993)]{Bertsimas1993}
D.~J. Bertsimas and G.~van Ryzin.
\newblock Stochastic and dynamic vehicle routing in the {E}uclidean plane with multiple capacitated vehicles.
\newblock \emph{Operations Research}, 41:\penalty0 60--76, 1993.

\bibitem[Borodin and El-Yaniv(1998)]{Borodin1998}
A.~Borodin and R.~El-Yaniv.
\newblock \emph{{Online Computation and Competitive Analysis}}.
\newblock Cambridge University Press, 1998.

\bibitem[Boysen et~al.(2019)Boysen, {D}e Koster, and Weidinger]{Boysen2019}
N.~Boysen, R.~{D}e Koster, and F.~Weidinger.
\newblock Warehousing in the e-commerce era: A survey.
\newblock \emph{European Journal of Operational Research}, 277\penalty0 (2):\penalty0 396--411, 2019.

\bibitem[Bukchin et~al.(2012)Bukchin, Khmelnitsky, and Yakuel]{Bukchin2012}
Y.~Bukchin, E.~Khmelnitsky, and P.~Yakuel.
\newblock Optimizing a dynamic order-picking process.
\newblock \emph{European Journal of Operational Research}, 219\penalty0 (2):\penalty0 335--346, 2012.

\bibitem[Cals et~al.(2021)Cals, Zhang, Dijkman, and van Dorst]{Cals2021}
B.~Cals, Y.~Zhang, R.~Dijkman, and C.~van Dorst.
\newblock Solving the online batching problem using deep reinforcement learning.
\newblock \emph{Computers \& Industrial Engineering}, 156:\penalty0 107221, 2021.

\bibitem[Chen et~al.(2010)Chen, Gong, {de Koster}, and {van Nunen}]{Chen2010}
C.-M. Chen, Y.~Gong, R.~{de Koster}, and J.~{van Nunen}.
\newblock A flexible evaluative framework for order picking systems.
\newblock \emph{Production and Operations Management}, 19:\penalty0 70--82, 2010.

\bibitem[Chen et~al.(2015)Chen, Cheng, Chen, and Chan]{Chen2015}
T.-L. Chen, C.-Y. Cheng, Y.-Y. Chen, and L.-K. Chan.
\newblock An efficient hybrid algorithm for integrated order batching, sequencing and routing problem.
\newblock \emph{International Journal of Production Economics}, 159:\penalty0 158--167, 2015.

\bibitem[Corder and Foreman(2009)]{Corder2009}
G.~Corder and D.~Foreman.
\newblock \emph{Comparing Variables of Ordinal or Dichotomous Scales: Spearman Rank-Order, Point-Biserial, and Biserial Correlations}, chapter~7, pages 122--154.
\newblock John Wiley \& Sons, Ltd, 2009.

\bibitem[Dauod and Won(2022)]{Dauod2022}
H.~Dauod and D.~Won.
\newblock Real-time order picking planning framework for warehouses and distribution centres.
\newblock \emph{International Journal of Production Research}, 60\penalty0 (18):\penalty0 5468--5487, 2022.

\bibitem[D'Haen et~al.(2023)D'Haen, Braekers, and Ramaekers]{DHaen2022}
R.~D'Haen, K.~Braekers, and K.~Ramaekers.
\newblock Integrated scheduling of order picking operations under dynamic order arrivals.
\newblock \emph{International Journal of Production Research}, 61\penalty0 (10):\penalty0 3205--3226, 2023.

\bibitem[Feigin(1979)]{Feigin1979}
P.~D. Feigin.
\newblock On the characterization of point processes with the order statistic property.
\newblock \emph{Journal of Applied Probability}, 16\penalty0 (2):\penalty0 297--304, 1979.

\bibitem[Fiat and Woeginger(1998)]{woegingerandfiat1998}
A.~Fiat and G.~J. Woeginger.
\newblock \emph{Online Algorithms}, volume 1442.
\newblock Springer: Berlin, 1998.

\bibitem[Giannikas et~al.(2017)Giannikas, Lu, Robertson, and McFarlane]{Giannikas2017}
V.~Giannikas, W.~Lu, B.~Robertson, and D.~McFarlane.
\newblock An interventionist strategy for warehouse order picking: Evidence from two case studies.
\newblock \emph{International Journal of Production Economics}, 189:\penalty0 63--76, 2017.

\bibitem[Henn(2012)]{Henn2012}
S.~Henn.
\newblock Algorithms for on-line order batching in an order picking warehouse.
\newblock \emph{Computers \& Operations Research}, 39\penalty0 (11):\penalty0 2549--2563, 2012.

\bibitem[Hwang and Jaillet(2018)]{Hwang2017}
D.~Hwang and P.~Jaillet.
\newblock Online scheduling with multi-state machines.
\newblock \emph{Networks}, 71\penalty0 (3):\penalty0 209--251, 2018.

\bibitem[Jaillet and Wagner(2008)]{Jaillet2008}
P.~Jaillet and M.~R. Wagner.
\newblock Generalized online routing: New competitive ratios, resource augmentation, and asymptotic analyses.
\newblock \emph{Operations Research}, 56\penalty0 (3):\penalty0 745--757, 2008.

\bibitem[Le-Duc and {D}e Koster(2007)]{LeDuc2007}
T.~Le-Duc and R.~M. {D}e Koster.
\newblock Travel time estimation and order batching in a 2-block warehouse.
\newblock \emph{European Journal of Operational Research}, 176\penalty0 (1):\penalty0 374--388, 2007.

\bibitem[Le~Gall(2022)]{LeGall2022}
J.-F. Le~Gall.
\newblock \emph{Measure Theory, Probability, and Stochastic Processes}, volume 295.
\newblock Springer Nature, 2022.

\bibitem[Li et~al.(2022)Li, Zhang, and Jiang]{Li2022}
Y.~Li, R.~Zhang, and D.~Jiang.
\newblock Order-picking efficiency in e-commerce warehouses: A literature review.
\newblock \emph{Journal of Theoretical and Applied Electronic Commerce Research}, 17\penalty0 (4):\penalty0 1812--1830, 2022.

\bibitem[Liu and Yu(2000)]{Liu2000}
Z.~Liu and W.~Yu.
\newblock Scheduling one batch processor subject to job release dates.
\newblock \emph{Discrete Applied Mathematics}, 105\penalty0 (1-3):\penalty0 129--136, 2000.

\bibitem[L{\"o}ffler et~al.(2022)L{\"o}ffler, Boysen, and Schneider]{Loeffler2022}
M.~L{\"o}ffler, N.~Boysen, and M.~Schneider.
\newblock Picker routing in {AGV}-assisted order picking systems.
\newblock \emph{INFORMS Journal on Computing}, 34\penalty0 (1):\penalty0 440--462, 2022.

\bibitem[L{\"o}ffler et~al.(2023)L{\"o}ffler, Boysen, and Schneider]{Loeffler2023}
M.~L{\"o}ffler, N.~Boysen, and M.~Schneider.
\newblock Human-robot cooperation: Coordinating autonomous mobile robots and human order pickers.
\newblock \emph{Transportation Science}, 57\penalty0 (4):\penalty0 839--1114, 2023.

\bibitem[Lorenz et~al.(2024)Lorenz, Otto, and Gendreau]{Lorenz2024}
C.~Lorenz, A.~Otto, and M.~Gendreau.
\newblock Exact solution approaches for the order batching, sequencing, and picker routing problem with release times.
\newblock \emph{Working paper series of the chair of Management Science, Operations- and Supply Chain Management of the University of Passau}, 2024.
\newblock URL \url{https://www.wiwi.uni-passau.de/en/management-science/research/publications}.

\bibitem[Lu et~al.(2016)Lu, McFarlane, Giannikas, and Zhang]{Lu2016}
W.~Lu, D.~McFarlane, V.~Giannikas, and Q.~Zhang.
\newblock An algorithm for dynamic order-picking in warehouse operations.
\newblock \emph{European Journal of Operational Research}, 248\penalty0 (1):\penalty0 107--122, 2016.

\bibitem[Marchet et~al.(2015)Marchet, Melacini, and Perotti]{Marchet2015}
G.~Marchet, M.~Melacini, and S.~Perotti.
\newblock Investigating order picking system adoption: a case-study-based approach.
\newblock \emph{International Journal of Logistics Research and Applications}, 18\penalty0 (1):\penalty0 82--98, 2015.

\bibitem[Napolitano(2012)]{napolitano2012}
M.~Napolitano.
\newblock 2012 warehouse/{DC} operations survey: mixed signals.
\newblock \emph{{Logistics Management (Highlands Ranch, Colo.: 2002)}}, 51\penalty0 (11), 2012.

\bibitem[Pardo et~al.(2024)Pardo, Gil-Borr{\'a}s, Alonso-Ayuso, and Duarte]{Pardo2023}
E.~G. Pardo, S.~Gil-Borr{\'a}s, A.~Alonso-Ayuso, and A.~Duarte.
\newblock Order batching problems: {T}axonomy and literature review.
\newblock \emph{European Journal of Operational Research}, 313\penalty0 (1):\penalty0 1--24, 2024.

\bibitem[Schiffer et~al.(2022)Schiffer, Boysen, Klein, Laporte, and Pavone]{Schiffer2022}
M.~Schiffer, N.~Boysen, P.~S. Klein, G.~Laporte, and M.~Pavone.
\newblock Optimal picking policies in e-commerce warehouses.
\newblock \emph{Management Science}, 68\penalty0 (10):\penalty0 7497--7517, 2022.

\bibitem[Schleyer and Gue(2012)]{Schleyer2012}
M.~Schleyer and K.~Gue.
\newblock Throughput time distribution analysis for a one-block warehouse.
\newblock \emph{Transportation Research Part E: Logistics and Transportation Review}, 48\penalty0 (3):\penalty0 652--666, 2012.

\bibitem[Ulmer(2017)]{Ulmer2017}
M.~W. Ulmer.
\newblock \emph{{Approximate Dynamic Programming for Dynamic Vehicle Routing}}, volume~61.
\newblock Springer, 2017.

\bibitem[Valle et~al.(2017)Valle, Beasley, and Da~Cunha]{Valle2017}
C.~A. Valle, J.~E. Beasley, and A.~S. Da~Cunha.
\newblock Optimally solving the joint order batching and picker routing problem.
\newblock \emph{European Journal of Operational Research}, 262\penalty0 (3):\penalty0 817--834, 2017.

\bibitem[Van~Gils et~al.(2019)Van~Gils, Caris, Ramaekers, and Braekers]{VanGils2019}
T.~Van~Gils, A.~Caris, K.~Ramaekers, and K.~Braekers.
\newblock Formulating and solving the integrated batching, routing, and picker scheduling problem in a real-life spare parts warehouse.
\newblock \emph{European Journal of Operational Research}, 277\penalty0 (3):\penalty0 814--830, 2019.

\bibitem[Van~Nieuwenhuyse and {D}e Koster(2009)]{Vannieuwenhuyse2009}
I.~Van~Nieuwenhuyse and R.~B. {D}e Koster.
\newblock Evaluating order throughput time in 2-block warehouses with time window batching.
\newblock \emph{International Journal of Production Economics}, 121\penalty0 (2):\penalty0 654--664, 2009.

\bibitem[Vanheusden et~al.(2023)Vanheusden, van Gils, Ramaekers, Cornelissens, and Caris]{Vanheusden2022}
S.~Vanheusden, T.~van Gils, K.~Ramaekers, T.~Cornelissens, and A.~Caris.
\newblock {Practical factors in order picking planning: State-of-the-art classification and review}.
\newblock \emph{International Journal of Production Research}, 61\penalty0 (6):\penalty0 2032--2056, 2023.

\bibitem[Xie et~al.(2023)Xie, Li, and Luttmann]{Xie2023}
L.~Xie, H.~Li, and L.~Luttmann.
\newblock Formulating and solving integrated order batching and routing in multi-depot {AGV}-assisted mixed-shelves warehouses.
\newblock \emph{European Journal of Operational Research}, 307\penalty0 (2):\penalty0 713--730, 2023.

\bibitem[Yu and {D}e Koster(2008)]{Yu2008}
M.~Yu and R.~{D}e Koster.
\newblock Performance approximation and design of pick-and-pass order picking systems.
\newblock \emph{IIE Transactions}, 40\penalty0 (11):\penalty0 1054--1069, 2008.

\bibitem[Yu and {D}e Koster(2009)]{Yu2009}
M.~Yu and R.~B. {D}e Koster.
\newblock The impact of order batching and picking area zoning on order picking system performance.
\newblock \emph{European Journal of Operational Research}, 198\penalty0 (2):\penalty0 480--490, 2009.

\bibitem[Zhang et~al.(2001)Zhang, Cai, and Wong]{Zhang2001}
G.~Zhang, X.~Cai, and C.~Wong.
\newblock On-line algorithms for minimizing makespan on batch processing machines.
\newblock \emph{Naval Research Logistics}, 48\penalty0 (3):\penalty0 241--258, 2001.

\bibitem[Zhang et~al.(2016)Zhang, Wang, and Huang]{Zhang2016}
J.~Zhang, X.~Wang, and K.~Huang.
\newblock Integrated on-line scheduling of order batching and delivery under {B2C} e-commerce.
\newblock \emph{Computers \& Industrial Engineering}, 94:\penalty0 280--289, 2016.

\bibitem[Zhang et~al.(2017)Zhang, Wang, Chan, and Ruan]{Zhang2017}
J.~Zhang, X.~Wang, F.~T. Chan, and J.~Ruan.
\newblock On-line order batching and sequencing problem with multiple pickers: A hybrid rule-based algorithm.
\newblock \emph{Applied Mathematical Modelling}, 45:\penalty0 271--284, 2017.

\bibitem[{\v{Z}}ulj et~al.(2022){\v{Z}}ulj, Salewski, Goeke, and Schneider]{Zulj2022}
I.~{\v{Z}}ulj, H.~Salewski, D.~Goeke, and M.~Schneider.
\newblock Order batching and batch sequencing in an {AMR}-assisted picker-to-parts system.
\newblock \emph{European Journal of Operational Research}, 298\penalty0 (1):\penalty0 182--201, 2022.

\end{thebibliography}


\newpage
\setcounter{page}{1}
\section*{\Large{Electronic Companion}}
\renewcommand{\thefigure}{\Roman{figure}}
\setcounter{figure}{0}
\setcounter{section}{0}
\renewcommand\thesection{\Alph{section}}
\section{Proof of Lemma \ref{prop:cover}} \label{sec:Proof_Lemma_warehouse_traversal}

Let us call an intersection of an outer aisle and an outer cross-aisle a \textit{corner}. The shortest time to traverse all the picking locations of the warehouse starting from the given position $s$ and ending in the given position $t$ cannot be larger than the time of the following route (see Figure~\ref{fig:UB_batch_cover}): Move from $s$ to some corner of the warehouse, visit all the picking locations following an S-shape route starting from this \textit{(first)} corner and ending in the respective \textit{last} corner, after that move to $t$. Obviously, the S-shape route to visit all the picking locations of the warehouse takes time $(a\cdot W+L)$. By a smart selection of the \textit{first} corner, where we start this S-shape route,  the total time $\zeta$ to reach this first corner from $s$ and then reach $t$ from the \textit{last} corner, does not exceed $(W+L)$.

Let $w_s^{min}$ and $w_t^{min}$ be the shortest horizontal distances from $s$ and $t$ to reach an outer cross-aisle, respectively (see Figure~\ref{fig:UB_batch_cover}). Observe that $w_s^{min}\leq 0.5\cdot W$ and $w_t^{min}\leq 0.5\cdot W$. If $w_s^{min}\leq w_t^{min}$, the picker shall move from $s$  along  $w_s^{min}$, then:
\begin{align}
    & \text{vertical component of }\zeta \nonumber 
    \\\leq & \ w_s^{min}+\max\{W-w_t^{min}, w_t^{min} \}\leq W
\end{align}
If $w_s^{min}\geq w_t^{min}$, the picker shall select the first corner such, that she can move from the last corner to $t$ along $w_t^{min}$, then:
\begin{align}
    & \text{vertical component of }\zeta \nonumber \\
     \leq & \ \max\{W-w_s^{min}, w_s^{min} \}+ w_t^{min}\leq W.
\end{align}

\begin{figure}[b]
    \centering
    \includegraphics[scale=0.5]{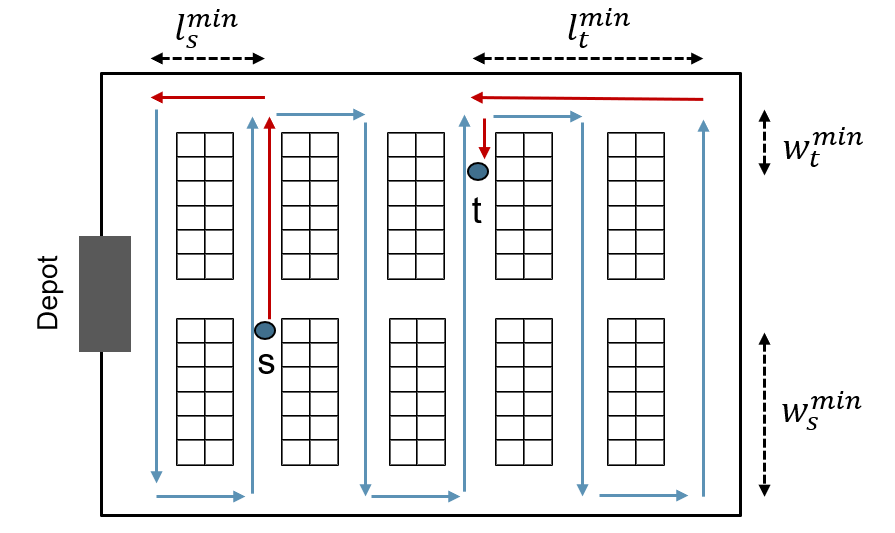}
    \caption{Warehouse traversal starting  in $s$ and ending in $t$\label{fig:UB_batch_cover}
    }
\end{figure}

By examining $l_s^{min}$ and $l_t^{min}$, which are the shortest horizontal distances from $s$ and $t$ to reach an outer aisle, respectively, we receive that the horizontal component of $\zeta\leq L$.

\section{Proof of Theorem~\ref{th:a.s.Poisson}} \label{sec:Proof_Poisson_a.s.}
Recall that $\vert Q^{Reopt}(I(n)^{rand}\vert $ is the queue length of available orders for \textit{Reopt} at the arrival of the last order $R_n$. The assumption $ \lim_{n\rightarrow \infty} \frac{\vert Q^{Reopt}(I(n)^{rand})\vert}{n} = 0 $ a.s. implies:
\begin{align}
  \frac{Reopt(I(n)^{rand})-R_n}{n} 
 \leq  \left(\frac{\vert Q^{Reopt}(I(n)^{rand})\vert}{c} +1\right) \cdot \frac{u}{n} \xrightarrow[]{n\rightarrow \infty} 0 \ \ \text{ a.s.}  \label{eq:Gn_limit} 
\end{align} 
where $u$ is the warehouse traversal constant from Lemma~\ref{prop:cover}. The right-hand side refers to a feasible picking plan, when the cart is fully packed with $c$ orders in each batch (if possible) and, to collect a batch, the picker traverses all the picking locations of the warehouse.
Since the inter-arrival times of the customer orders $X_i, i \in \{1,...,n\} $ are i.i.d., we can apply the Strong Law of Large Numbers:
\begin{align}
   \frac{R_n}{n}=\frac{\sum\limits_{i=1}^n X_i}{n} \xrightarrow[]{n\rightarrow \infty} \frac{1}{\lambda} < \infty  \ \ \text{a.s.} \label{eq:Rn_convergence}
\end{align}
Putting (\ref{eq:Gn_limit}) and (\ref{eq:Rn_convergence}) together:
\begin{align}
     \frac{Reopt(I(n)^{rand})}{ CIOPT(I(n)^{rand})}   
    \leq & \ \frac{R_n+Reopt(I(n)^{rand})-R_n}{R_n} \\
    = & \  1+ \frac{Reopt(I(n)^{rand})-R_n}{n}\cdot\frac{n}{R_n} \\
    & \ \xrightarrow[]{n\rightarrow \infty} 1 \qquad {a.s.}
\end{align}

\section{Proof of Theorem ~\ref{th:stoch_optimality}} \label{sec:Proof_stoch_opt}
This proof uses the Order Statistics Property of the homogeneous Poisson Process (cf. \citet{Feigin1979}):
\begin{proposition}\label{prop:OS_property}
Suppose that order arrival times $R_1,R_2, ...$ follow a homogeneous Poisson process with rate $\lambda$. Let $N$ be the number of arrival times in the given time interval $[0,t]$. Then, conditional on $N=n$, the successive arrival times $(R_1,...,R_n)$ are distributed as the order statistics of $n$ independent identically distributed random variables $(Y_i)_{i \in \mathbb{N}}$ with a uniform distribution on $[0,t]$.

\end{proposition}
We use Proposition~\ref{prop:OS_property} to extend the analysis of Theorem~\ref{th:stoch_optimality} from the probability space described in Section~\ref{sec:orderstat} to the probability space of Section~\ref{sec:poisson}. Thereby, we use the following notation:
\begin{itemize}
\item Probability measures $\mathbb{P}_{os,unif(t)}$ for $t\in \mathbb{R}^+$, when the arrival times are independent uniformly$[0,t]$ -distributed random variables. Observe that in this case, order arrival times satisfy the order statistics property and, therefore, satisfy the conditions for Theorem~\ref{th:asymp_opt_makesp}. 
\item Probability measures $\mathbb{P}_{Pois(\lambda)}$ for $\lambda\in \mathbb{R}^+$, when  the arrival times $R=(R_i)_{i\in\mathbb{N}}$ follow a Poisson Point process with rate $\lambda$.  
\end{itemize}

We denote the optimality ratio of \textit{Reopt} for any given random instance $I^{rand}$ in a  given picking system (\textit{Pcart-N}, \textit{Pcart} or \textit{Robot}) as the respective measurable function $f: \Omega \rightarrow [1,\infty), I^{rand} \mapsto \frac{Reopt(I^{rand})}{CIOPT(I^{rand})}$, where $\Omega$ is the instance space.

The proof proceeds in four steps.

 First, we start with independently and uniformly $[0,t]$ -distributed order arrivals. The \textit{almost sure (a.s.)} asymptotic convergence of $f(I(n)^{rand})$ towards $1$ (cf. Theorem~\ref{th:asymp_opt_makesp}) directly implies the \textit{convergence in probability} of $f(I(n)^{rand})$ for $n\rightarrow \infty$ (cf. Proposition~10.3 in \citet{LeGall2022}). 
In \textit{Pcart-N}, \textit{Pcart} and \textit{Robot}, the ratio $f(I(n)^{rand})$ is bounded for all $n\in \mathbb{N}$ (see Proposition~\ref{prop:strict_CR}), so is the expected value $\mathbb{E}_{os,unif(t)}[f(I(n)^{rand}]\leq \sigma(Reopt)< \infty$. Therefore, the convergence in probability implies the \textit{convergence in mean} (see Proposition~10.4 in \citet{LeGall2022}):
\begin{align}
    \lim_{n\rightarrow \infty} \mathbb{E}_{os,unif(t)} [f(I(n)^{rand})]=1 \qquad \forall t \in \mathbb{R}^+\label{eq:L1_conv}
\end{align}

Second, suppose the arrival times follow a Poisson process with some rate $\lambda\in \mathbb{R}^+$ and let $I(t)^{rand}$ consist of all orders that arrive in the fixed time interval $[0,t], t\in \mathbb{R}^+$. Let denote the number of these orders as $N(t)\in \mathbb{N}$ and observe that $N(t)$ a Poisson-($\lambda\cdot t)$-distributed random variable. Then by Proposition~\ref{prop:OS_property}:
\begin{align}
     \mathbb{E}_{Pois(\lambda)}[f(I(t)^{rand}) \ \vert N(t)=n] 
    = \ \mathbb{E}_{os,unif(t)}[f(I(n)^{rand})]
\end{align}
Thus, for any $\lambda\in \mathbb{R}^+$ and any $t\in \mathbb{R}^+$, by (\ref{eq:L1_conv}): 
\begin{align}
    \lim_{n\rightarrow \infty} \mathbb{E}_{Pois(\lambda)} [f(I(t)^{rand}) \ \vert \ N(t)=n]=1 \label{eq:L1_conv_Poisson}
\end{align}

The convergence in (\ref{eq:L1_conv_Poisson}) implies that for any $\epsilon>0$, there exists $n_{\epsilon}\in \mathbb{N}$ such that
\begin{align}
\mathbb{E}_{Pois(\lambda)} [f(I(t)^{rand}) \ \vert \ N(t) = n] \leq 1 + \frac{\epsilon}{2} \qquad \forall n > n_{\epsilon}   
\end{align}

Third, by applying the Law of Total Expectation, we compute the \textit{unconditional} mean by examining groups of instances with $N(t)>n_{\epsilon}$ and with $N(t)\leq n_{\epsilon}$:
\begin{align}
      \mathbb{E}_{Pois(\lambda)} [f(I(t)^{rand}) ] 
     = & \  \mathbb{P}_{Pois(\lambda)}(N(t) \leq n_{\epsilon}) 
      \cdot  \ \mathbb{E}_{Pois(\lambda)} [f(I(t)^{rand}) \ \vert \ N(t) \leq n_{\epsilon}] \nonumber \\ 
     & +  \  \sum\limits_{n=n_{\epsilon}+1}^{\infty} \mathbb{P}_{ Pois(\lambda)}(N(t) = n) 
      \cdot \ \mathbb{E}_{Pois(\lambda)} [f(I(t)^{rand}) \ \vert \ N(t) = n] 
     \label{eq:L1_conv_step1}
\end{align}
By Proposition~\ref{prop:strict_CR} on the strict competitive ratio, $\mathbb{E}_{Pois(\lambda)} [f(I(t)^{rand}) \ \vert \ N(t) \leq n_{\epsilon}]\leq\sigma(Reopt)<\infty$. Moreover,  
the probability mass $\mathbb{P}_{ Pois(\lambda)}(N(t) > n_{\epsilon})\leq 1$ by definition. Hence:
\begin{align}
     \mathbb{E}_{ Pois(\lambda)} [f(I(t)^{rand}) ] 
     \leq & \  \sigma(Reopt) \cdot \mathbb{P}_{ Pois(\lambda)}(N(t) \leq n_{\epsilon}) 
    +  \ (1+\frac{\epsilon}{2})\cdot \mathbb{P}_{ Pois(\lambda)}(N(t) > n_{\epsilon}) \\
          \leq &  \sigma(Reopt)  \cdot  \mathbb{P}_{ Pois(\lambda)}(N(t) \leq n_{\epsilon})  + (1 + \frac{\epsilon}{2}) \label{eq:eq:L!_conv_laststep}
\end{align}
Observe that $n_{\epsilon}$ does not depend on the arrival rate $\lambda$, but $\lambda$ influences the probability of instances with $N(t)\leq n_{\epsilon}$. Let $\lambda_i>0, i \in \mathbb{N},$ be any increasing sequence of arrival rates such that $\lambda_i \rightarrow \infty$ if $i \rightarrow \infty$. Recall that the cumulative distribution function of the Poisson distribution is the \textit{regularized gamma function} $Q$, i.e., $\mathbb{P}_{ Pois(\lambda)}(N(t) \leq n_{\epsilon})=Q(n_{\epsilon}, \lambda \cdot t)$. By definition of $Q$ for fixed $n_{\epsilon}$, there exists an $i_{\epsilon}\in \mathbb{N}$ such that for all $i\geq i_{\epsilon}: \sigma(Reopt)  \cdot \mathbb{P}_{ Pois(\lambda)}(N(t) \leq n_{\epsilon})= \sigma(Reopt)  \cdot Q(n_{\epsilon}, \lambda_i \cdot t)\ \leq \frac{\epsilon}{2} $.
After implementing this in (\ref{eq:eq:L!_conv_laststep}):
\begin{align}
    \mathbb{E}_{ Pois(\lambda_i)} [f(I(t)^{rand}) ] \leq 1 + \epsilon \qquad \forall i \geq i_{\epsilon} \label{eq:laststep}
\end{align}
Since (\ref{eq:laststep}) is valid for any $\epsilon>0$, we proved that for any increasing sequence $(\lambda_i)_{i\in \mathbb{N}}$ of arrival rates with $\lambda_i\rightarrow \infty $ if $i\rightarrow \infty$
\begin{align}
   \lim_{i\rightarrow \infty}  \mathbb{E}_{ Pois(\lambda_i)} [f(I(t)^{rand}) ] \rightarrow 1
\end{align}

Finally, the second convergence statement of Theorem~\ref{th:stoch_optimality} follows by Markov's inequality for the nonnegative random variable $(f(I(t)^{rand})-1)$ and any given $\epsilon>0$:
\begin{align}
     \mathbb{P}_{Pois(\lambda_i)}(f(I(t)^{rand}) \geq 1+ \epsilon) 
    \leq  \ \frac{ \mathbb{E}_{ Pois(\lambda_i)} [f(I(t)^{rand}) -1]}{\epsilon} 
    \rightarrow 0 \qquad \text{if }i\rightarrow \infty
\end{align}

\section{Auxiliary Lemma on the asymptotic competitive ratio }\label{sec:aux_lemma_asymp_cr}
\begin{lemma}\label{lem:lim_asymp_cr}
In \textit{Pcart} and \textit{Robot}, the following holds true:
\begin{align}
lim_{n\rightarrow \infty} \frac{Reopt(I(n))}{CIOPT(I(n))} = \alpha \\
\text{ together with } \quad \lim_{n\rightarrow \infty}CIOPT(I(n)) \rightarrow \infty
\end{align}
implies that $\sigma_{asymp.}(Reopt)\geq \alpha$.
\end{lemma}
\begin{proof}
In a proof by contradiction, suppose that for all $n\in \mathbb{N}$ and for all instances $I(n)$, $Reopt(I(n))\leq (\alpha-\epsilon) CIOPT(I(n)) +const$ for two constants $const, \epsilon >0$. Then $\frac{Reopt(I(n))}{CIOPT(I(n))}\leq (\alpha - \epsilon) + \frac{const}{CIOPT(I(n))}) $ for all $n\in \mathbb{N}$. But since $\frac{const}{CIOPT(I(n))}\xrightarrow{n \rightarrow \infty }0 $, this is a contradiction to  $lim_{n\rightarrow \infty}\frac{Reopt(I(n))}{CIOPT(I(n))}= \alpha$. 
\end{proof}
\section{Proof of Proposition~\ref{prop:asymp_cr_LBs}} \label{sec:proof_asyp_cr_LBs}

\begin{figure}[b]
    \centering
    \includegraphics[scale=0.35]{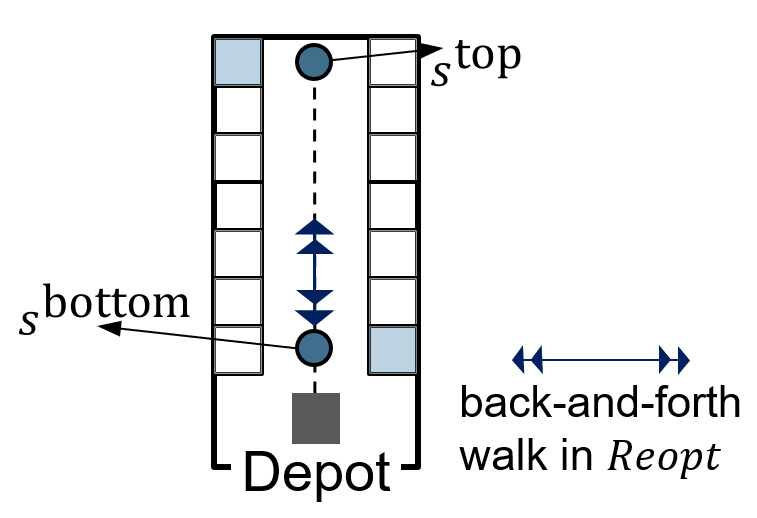}
    \caption{\textit{Robot}: Unfortunate large instance for \textit{Reopt} \label{fig:Asy_Robot_ex} \\ \scriptsize{ Note. $c=2$. Two-item orders, each with one item $s^\text{bottom}$ and one item $s^\text{top}$, arrive such that the picker in \textit{Reopt} keeps walking back and forth without completing any batch until the arrival of the last order.}}
\end{figure}

Figure~\ref{fig:Asy_Robot_ex} provides an unfortunate example $I^\text{worst}$ for \textit{Robot} with $c=2$ in a one-aisle warehouse; the example can be generalized for $c>2$ and more general warehouse layouts. 

Arriving orders in $I^\text{worst}$ consist of two items each: $s^\text{bottom}$ (placed at the bottom of the aisle) and $s^\text{top}$ (placed on the top of the aisle). Recall that $W$ is the width of the warehouse and $l^d$ denotes the location of the depot. Let  $d(l_d,s^\text{bottom})=2\epsilon$, $(l_d,s^\text{top})=W$, and $(s^\text{bottom},s^\text{top})=W-2\epsilon$. 

Instance $I^\text{worst}$ consists of $n=2k+1$ orders, $k\in \mathbb{N}$. The idea is to set the inter-arrival times of the orders such that \textit{Reopt} makes the picker oscillate between $s^\text{bottom}$ and the middle of the aisle without picking any item until the last order arrives.

The first order arrives at $r_1=2\epsilon$, the picker in \textit{Reopt} picks item $s^{\text{bottom}}$ of order $o_1$ at time $4 \epsilon$. 

At time $r_2=\frac{1}{2} W+3\epsilon - \frac{\delta}{2n-2}$, $o_2$ arrives and the picker in \textit{Reopt} is at distance  $\frac{1}{2}W-\epsilon- \frac{\delta}{2n-2}$ from $s^{\text{bottom}}$, with $s^{\text{bottom}}$ of order $o_1$ placed in her cart. \textit{Reopt} has two choices. The first one is to return to $s^{\text{bottom}}$ and batch the new order $o_2$ together with $o_1$ at the total cost of $r_2+\frac{1}{2}W-\epsilon- \frac{\delta}{2n-2}+W-2\epsilon=2W-2\cdot \frac{\delta}{2n-2}$. The second one is to continue to $s^{\text{top}}$ and pick $o_1$ and $o_2$ in separate batches at the total cost of $2\epsilon+W+(W-2\epsilon)=2W$. \textit{Reopt} decides for the first option. 

At time $r_3=W+2\epsilon-3\frac{\delta}{2n-2}$, $o_3$ arrives and the picker in \textit{Reopt} is at distance $\frac{\delta}{2n-2}$ from $s^{\text{bottom}}$. Now  \textit{Reopt} has two choices. The first one is to move up and finish order $o_1$ in a separate batch by picking $s^{\text{top}}$, then collect orders $o_2$ and $o_3$ together at the total cost of $r_3+W-2\epsilon-\frac{\delta}{2n-2}+W-2\epsilon=3W-2\epsilon-4\frac{\delta}{2n-2}$. The second one is to collect item  $s^{\text{bottom}}$ of $o_2$, finish batch $S_2=\{o_1,o_2\}$ and pick $o_3$ separately at the total cost of $r_3+\frac{\delta}{2n-2}+2(W-2\epsilon)=3W-2\epsilon-2\frac{\delta}{2n-2}$. \textit{Reopt} will decide for the first option.  

In general, at time $r_{2p}= \frac{2p-1}{2}W-(2p-5)\cdot\epsilon-(4p-3)\cdot \frac{\delta}{2n-2}, p\in \{2,3,...,k\}$, the picker in \textit{Reopt} is at distance $\frac{1}{2}W-\epsilon- \frac{\delta}{2n-2}$ from $s^{\text{bottom}}$. Now  \textit{Reopt} has two choices. The first one is to return to $s^{\text{bottom}}$ to batch the new order $o_{2p}$ with the commenced order $o_1$, then pick all the remaining orders in pairs. 
The second one is to continue to $s^{\text{top}}$ to finish the commenced batch $S_1=\{o_1\}$, then pick the remaining orders in $p-1$ pairs and the last order in a separate batch. 
\textit{Reopt} decide for the first option with the total cost of $r_{2p}+\frac{1}{2}W-\epsilon-\frac{\delta}{2n-2}+p(W-2\epsilon)=2pW-(4p-4)\epsilon-(4p-2)\frac{\delta}{2n-2}$. 

At time $r_{2p+1}= pW-(2p-4)\epsilon-(4p-1)\cdot \frac{\delta}{2n-2}, p\in \{2,3,...,k\}$, the picker is at distance $\frac{\delta}{2n-2}$ from $s^{\text{bottom}}$. \textit{Reopt} has two choices. The first one is to change the course and move up to $s^{\text{top}}$ to finish order $o_1$ in a separate batch, then pick the rest of the orders in pairs. The second one is to pick up item $s^{\text{bottom}}$ of $o_2$, finish batch $S_1=\{o_1,o_2\}$, then pick all orders in pairs except from batch $S_{p+1}=\{o_{2p+1}\}$, which is picked separately. \textit{Reopt} will decide for the first option at the total cost of $r_{2p+1}+W-2\epsilon-\frac{\delta}{2n-2}+p(W-2\epsilon)=(2p+1)W-(4p-2)\epsilon-4p\cdot \frac{\delta}{2n-2}$. 

Overall, to pick $n=2k+1$ orders of $I^\text{worst}$, \textit{Reopt} requires $Reopt(I^\text{worst})=(2k+1)W-(4k-2)\epsilon-4k\cdot \frac{\delta}{2n-2}=(2k+1)W-(4k-2)\epsilon-\delta$ time.  $CIOPT$ will pick the first order in a separate batch, which is completed at time $W$, and collect the remaining orders in pairs, so that $ CIOPT(I^\text{worst})=W+k\cdot (W-2\epsilon)=(k+1)W-2k\epsilon$. 
It follows that $\frac{Reopt(I^\text{worst})}{CIOPT(I^\text{worst}))}=2-\frac{W-2 \epsilon + \delta}{CIOPT(I^\text{worst})}\xrightarrow{k\rightarrow \infty}2 $ as $CIOPT(I^\text{worst})\xrightarrow{k\rightarrow \infty}\infty $. This on the other hand, implies that $\sigma_{asymp}(Reopt)\geq 2$ by Lemma~\ref{lem:lim_asymp_cr}.

\section{Proof of Proposition~\ref{prop:LB_abs_CR_general}}\label{sec:proof_LB_abs_general}

\begin{figure} [b]
    \centering
    \begin{subfigure}{0.49\textwidth}
        \includegraphics[width=0.9\textwidth]{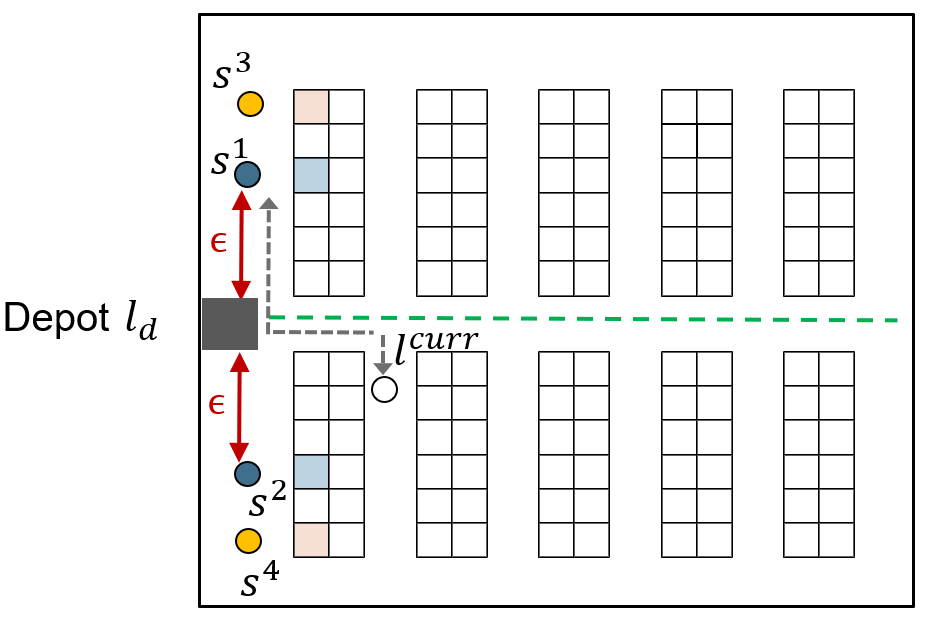}
        \caption*{ (a) Warehouse with the depot \textit{not} located in the corner}
        \label{fig:depot_in_aisle}
    \end{subfigure}
    ~ 
    \begin{subfigure}{0.48\textwidth}
        \includegraphics[width=0.91\textwidth]{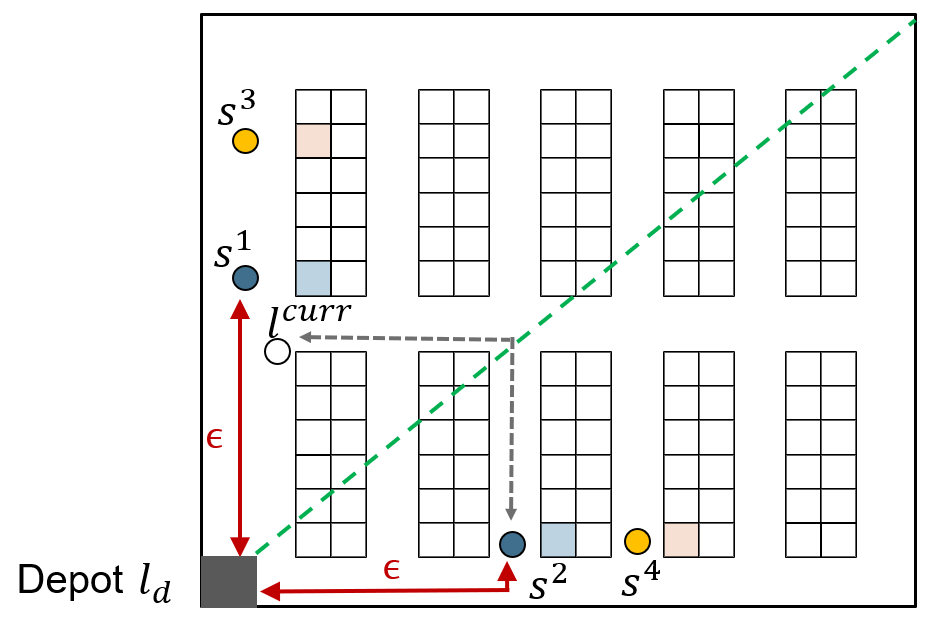}
        \caption*{ (b) Warehouse with the depot in the corner}
        \label{fig:depot_in_corner}
    \end{subfigure}
    \caption{Positioning of the potential order locations $s^1$ and $s^2$ in Proposition~\ref{prop:LB_abs_CR_general} for \textit{Robot} and \textit{Pcart} \label{fig:bad_all_algs_makesp} \\ \scriptsize{ Note. Observe that $d(s^1,l_d)=d(s^2,l_d)=\epsilon$ for some $\epsilon$, which depends on the location of $l^d$. The dotted green line separates the warehouse into two sections. For each location $l^{curr}$ in the same section as $s^1$, $d(l^{curr},s^2)\geq \epsilon$ and for each location $l^{curr}$ in the same section as $s^2$, $d(l^{curr},s^1)\geq \epsilon$.}}
\end{figure}

Theorem 3.1 for N-TSP and Theorem 3.3 for H-TSP by \citet{Ausiello1999} consider the case when requests are located on the \textit{real line}.

What remains is the translation of the specific positioning of the picking locations $s^1$ and $s^2$, from the real line to the warehouse, such that:
\begin{itemize}
\item They are at equal distance to the depot $d(l_d,s^1)=d(l_d,s^2)=\epsilon$.
\item \textit{Any} other location in the warehouse is at a larger distance to one of the picking locations:
\begin{align}
\max\{d(l^{curr},s^1), d(l^{curr},s^2)\}\geq \epsilon \label{eq:main_anyalg}
\end{align}
Observe that $d()$ refers to the warehouse metric as explained in Section~\ref{sec:problem}.
\end{itemize}

Figure~\ref{fig:bad_all_algs_makesp} explains, how to position $s^1$ and $s^2$ if  the depot $l^d$ is located \textit{i)} in a corner of the warehouse (Figure~\ref{fig:bad_all_algs_makesp}a) or \textit{ii)} in any other location (Figure~\ref{fig:bad_all_algs_makesp}b). 

To prove (\ref{eq:main_anyalg}), compare $d()$ with the Manhattan metric $m()$:
\begin{itemize}
    \item Obviously $d(l,l')\geq m(l,l')$ for any locations $l, l'$ in the warehouse.
    \item The green dotted line in Figure~\ref{fig:bad_all_algs_makesp} depicts locations $l$ that are \textit{Manhattan-equidistant} to $s^1$ and $s^2$ at a distance of \textit{at least} $\epsilon$: $m(l,s^1)=m(l,s^2)\geq \epsilon$.
    \item The green dotted line divides the warehouse into two \textit{sections}:  $s^1$ is in section 1 and $s^2$ is in section 2. If the picker's location $l^{curr}$ is in section 1, the picker must cross the line to reach $s^2$ and vice-versa. Thus, for any location  $l^{curr}$ in the warehouse, $\max\{d(l^{curr},s^1), d(l^{curr},s^2)\}\geq \max\{m(l^{curr},s^1), m(l^{curr},s^2)\}\geq\epsilon$.
\end{itemize}

\textbf{The case of  \textit{Robot}}. The main idea of Theorem 3.1 for N-TSP is that the upcoming request (which we can interpret as a single-item order) is located in one of two locations -- $s^1$ and $s^2$ -- placed on opposite sides of the depot $l^d$ at distance $d(l_d,s^1)=d(l_d,s^2)=\epsilon$. For any picker position $l^{curr}$ on the \textit{real line}, $\max\{d(l^{curr},s^1), d(l^{curr},s^2)\}\geq \epsilon$. Now, consider any algorithm $ALG$ and the respective picker location $l^{curr}$ at time $r_1$. There is an instance $I^{worst}$ defined on the \textit{real line}, in which request $o_1$ appears to the left of $l_d$ if the picker's position $l^{curr}$ is to the right of it, and \textit{vice versa}; i.e., $ALG$ is on the `wrong side' of the depot in this instance. $CIOPT$ has complete information about the location of $o_1$ and moves immediately in the right direction.  In \textit{Robot}, an unfortunate instance $I^{worst}$ in a warehouse of any geometry, with $n=1$ order and $r_1=\epsilon$ can be constructed by the same scheme as just described for of N-TSP. Indeed, by this scheme, $CIOPT(I^{worst})=\epsilon$ and $ALG(I^{worst})=\epsilon + \max\{d(l^{curr},s^1), d(l^{curr}),s^2)\}\geq 2\epsilon$. 

\textbf{The case of \textit{Pcart}}. In \textit{Pcart}, the proof uses the same positional scheme for the picking locations $s^1$ and $s^2$, and is \textit{adapted} from the proof of Theorem 3.3 from \citet{Ausiello1999}. It constructs seven unfortunate instances $I^{w1}$ to $I^{w7}$ depending on the decisions of the deterministic algorithm $ALG$:
\begin{itemize}
\item Instances $I^{w1}$ and $I^{w2}$, each  with one ($n=1$) single-item order arriving at $r_1=\epsilon$. The order's position is at $s^1$ in $I^{w1}$ and $s^2$ in $I^{w2}$, respectively.
\item Instances $I^{w3}$ and $I^{w4}$, each with two ($n=2$) orders arriving at $r_1=\epsilon$ and $r_2=3\epsilon$. The first order has two items at positions $s^1$ and $s^2$ in both instances. The second order has one item at position $s^1$ in $I^{w3}$ and $s^2$ in $I^{w4}$, respectively.
\item Instance $I^{w5}$ with one ($n=1$) order arriving at $r_1=\epsilon$ with two items at positions $s^1$ and $s^2$. 
\item Instances $I^{w6}$ and $I^{w7}$ with two orders ($n=2$). The first order has the same two items $s^1$ and $s^2$ and arrival time $r_1=\epsilon$ as in $I^{w5}$. The second order $o_2$ arrives at time $r_2=(3+q)\epsilon$, where $0\leq q< 4\left(\frac{9+\sqrt{17}}{8}\right)-5<1.6$ and is computed as explained below. In $I^{w6}$, order $o_2$ has one item at position $s^3$ located in the direction of $s^1$ away from the depot with $d(s^3, l^d)=(1+q)\epsilon, d(s^3, s^1)=q\epsilon$ and $d(s^3, s^2)=(2+q)\epsilon$ (see Figure~\ref{fig:bad_all_algs_makesp}). In $I^{w7}$, the single item of order $o_2$ is at position $s^4$, which is constructed along the same lines as $s^3$, but in the direction of $s^2$ away from the depot.
\end{itemize}
In all these instances, $CIOPT$ has complete information. If $c\geq 2$, 
\begin{align}
&CIOPT(I^{w1})=CIOPT(I^{w2})=2\epsilon \label{eq:cioptw1}\\
&CIOPT(I^{w3})=CIOPT(I^{w4})=4\epsilon \label{eq:cioptw3}\\
&CIOPT(I^{w5})=4\epsilon \label{eq:cioptw5}\\
&CIOPT(I^{w6})=CIOPT(I^{w7})=(4+2q)\epsilon \label{eq:cioptw6}
\end{align}

In contrast, no algorithm $ALG$ can differentiate between the instances $I^{w1}$ to $I^{w7}$ before time $t=\epsilon$, between instances $I^{w3}$ to $I^{w7}$ before $t=3\epsilon$, and between instances $I^{w5}$ to $I^{w7}$ before $t=(3+q)\epsilon$. The case-by-case proof of \citet{Ausiello1999} shows that any deterministic $ALG$ reaches the ratio of $\frac{ALG(I)}{CIOPT(I)}\geq \frac{9+\sqrt{17}}{8}\sim 1.64$ for at least one of the instances described above.

In the proof, we keep track of the distance from the current picker position $l^{curr}$ in some arbitrary algorithm $ALG$ to the \textit{green dotted line} ($line$) (see Figure~\ref{fig:bad_all_algs_makesp}), which is the \textit{shortest} distance from $l^{curr}$ to any point on this line. 

Let us denote $\sigma(ALG)$ simply as $\sigma$. Assume that $ALG$ has $\sigma<(9+\sqrt{17})/8$. Then, in instances $I\in \{I^{w1}, I^{w2}\}$, the result of $ALG$ should be less than $\sigma\cdot CIOPT(I)=2\epsilon\sigma$ (see (\ref{eq:cioptw1})). To achieve this, at time $t=\epsilon$, the picker must be at a point $l^{curr}_1$ of distance $\leq (2 \sigma-3)\epsilon$ from $line$, since she has to cross the \textit{line} to pick the item $s^1$ (in $I^{w1}$) or $s^2$ (in $I^{w2}$) and then return to the depot, which takes her at least $\epsilon+d(line,l^{curr}_1)+\epsilon+\epsilon$ time (see (\ref{eq:main_anyalg})). 

Similarly, in instances $I\in \{I^{w3}, I^{w4}\}$, the result of $ALG$ should be less than $\sigma\cdot CIOPT(I)=4\epsilon\sigma$ (see (\ref{eq:cioptw3})). Observe that at $t=3 \epsilon$, given $l^{curr}_1$ and $r_{1,2}=\epsilon$, $ALG$ was not able to pick at least one of the items -- $s^1$ or $s^2$. \textit{W.l.o.g}, let it be $s^1$. Then at time $t=3 \epsilon$, the picker must be at a point $l^{curr}_2$ of distance $\geq (7- 4 \sigma)\epsilon$ from $line$ (see \citet{Ausiello1999} for more details). 

It remains to examine instances $I\in \{I^{w5}, I^{w6}, I^{w7}\}$. Recall that the picker's position $l^{curr}_2$ is situated at some distance $\geq(7- 4 \sigma)\epsilon$ from the $line$ at time $t=3 \epsilon$, thus can be shown that she has to cross the $line$ at least once to collect the remaining items as follows. Since, as explained above, the picker cannot have picked both items of the first order ($s^1$ and $s^2$) at $t=3 \epsilon$, there are two cases: 
\begin{itemize}
\item  The picker has picked neither $s^1$ nor $s^2$, then she will have to cross the $line$ at least once to serve the first order. 
\item $W.l.o.g.$, the picker has picked $s^2$. Then, if we look at both warehouse sections separated by the $line$, the picker has to be located at the part with $s^2$ from the line. Indeed, she started moving to $s^2$ at time $t=\epsilon$ being at the other side of the $line$ at a distance of at most $(2\sigma-3)\epsilon$. Even if she moved directly to $s^1$ after having visited $s^2$, at time $t=3\epsilon$, she is at a distance of at most $3\epsilon-\epsilon-(\epsilon +(2\sigma-3)\epsilon)=(\epsilon -(2\sigma-3)\epsilon)<\epsilon$ from $s^2$. Therefore, also in this case, the picker has to cross the $line$ to collect the remaining item of the first order.
\end{itemize}
To sum up, in both cases in $ALG$, the picker crosses the $line$ at some time after $t=3\epsilon$. Let denote this time $(3+q)\epsilon$. Since the result  of $ALG$ should be less than $\sigma\cdot CIOPT(I^{w5})=4\epsilon\sigma$, the picker should cross the $line$ at no later time than $(4\epsilon\sigma-2\epsilon)$. Therefore, we have
\begin{align}
q\leq 4\sigma-5 \label{eq:q}
\end{align}

Finally, if in $ALG$ the picker crosses the $line$  not later than time $(3+q)\epsilon=(4 \sigma -2)\epsilon$, see (\ref{eq:q}), then at least in one of the instances $I\in \{I^{w6}, I^{w7}\}$, $ALG(I)\geq (7+3q)\epsilon$ (see (\ref{eq:cioptw6})), whereas $CIOPT=(4+2q)\epsilon$. The least value of $\sigma$ that satisfies $\sigma\geq\frac{(7+3q)\epsilon}{(4+2q)\epsilon}$ given (\ref{eq:q}) is $\sigma=\frac{9+\sqrt{17}}{8}$.

\newpage 

\section{Average observed optimality ratios of \textit{Reopt} in \textit{Pcart} and \textit{Robot}} ~\label{sec:exp_ratios_Pcart}
Tables~\ref{tab:compratio_avg_Pcart} and~\ref{tab:compratio_avg} report the average observed optimality ratios of \textit{Reopt}, for the experiments described in Section~\ref{sec:experiments}.
\begin{table}
\caption{Average optimality ratios of \textit{Reopt} in \textit{Pcart} }  \label{tab:compratio_avg_Pcart}
\scriptsize{
\begin{tabular}{lrrrrrrrrrrrrr}
\toprule
Setting & \multicolumn{13}{c}{\underline{The number of orders $n=$}} \\[0.05cm]
  &  3 & 4 & 5 & 6 & 7 & 8 & 9 & 10 & 11 & 12 & 13 & 14 & 15\\
\midrule
Base & 1.11 & 1.07 & 1.08 & 1.06 & 1.07 & 1.06 & 1.07 & 1.05 & 1.07 & 1.06 & 1.05 & 1.06 & 1.06 \\
Largewarehouse & 1.13 & 1.11 & 1.09 & 1.09  & 1.08 & 1.07 & 1.08 & 1.10 & 1.06 & 1.06 & 1.08 & 1.08 & *1.06 \\
Lesscrossaisles & 1.11 & 1.10 & 1.07 & 1.07  & 1.06 & 1.07 & 1.07 & 1.05 & 1.07 & 1.04 & 1.05 & 1.07 & 1.07 \\
Classbaseddispersion & 1.11 & 1.07 & 1.08 & 1.06 & 1.08 & 1.06 & 1.07 & 1.05 & 1.03 & 1.04 & 1.05 & 1.05 & *1.04  \\ %
Smallorders & 1.16 & 1.11 & 1.09 & 1.07 & 1.07 & 1.05 & 1.06 & 1.04 & 1.04 & 1.04 & 1.06 & 1.05 & 1.05 \\ 
Largeorders & 1.10 & 1.07 & 1.07 & 1.04 & 1.06 & *1.05 & *1.08 & -- & -- & -- & -- & -- & -- \\ 
Smallbatches & 1.10 & 1.07 & 1.06 & 1.05 & 1.04 & 1.04 & 1.04 & 1.03 & 1.03 & 1.03 & 1.03 & 1.02 & 1.02 \\
Largebatches  & 1.10 & 1.09 & 1.12 & 1.09 & 1.07 & 1.06 & 1.08 & -- & -- & -- & -- & -- & --  \\ 
Smallrate & 1.10 & 1.11 & 1.08 & 1.06 & 1.05 & 1.06 & 1.03 & 1.04 & 1.04 & 1.06 & 1.04 & 1.06 & 1.05 \\ 
Largerate  & 1.13 & 1.09 & 1.09 & 1.08 & 1.06 & 1.06 & 1.07 & 1.07 & 1.07 & 1.07 & 1.07 & 1.06 & *1.04 \\ 
\midrule
Total  & 1.12 & 1.09 & 1.08 & 1.07 & 1.07 & 1.06 & 1.06 & 1.05 & 1.05 & 1.05 & 1.05 & 1.06 & 1.05 \\ 
\midrule[\heavyrulewidth]
\multicolumn{14}{l}{\scriptsize{\textit{Note.} $^*$ Results for less than 10 (out of 10) instances. CIOPT did not solve 1-2 instances to optimality.}}\\
\multicolumn{14}{l}{\scriptsize{\quad\quad~~ -- No results reported, since CIOPT solved less than 8 instances to optimality.}}
  \end{tabular}}
\end{table}

\begin{table}
\caption{Average observed optimality ratios of \textit{Reopt} in \textit{Robot} } \label{tab:compratio_avg}
\scriptsize{
\begin{tabular}{lrrrrrrrrrrrrr}
\toprule
Setting & \multicolumn{13}{c}{\underline{The number of orders $n=$}} \\[0.05cm]
  &  3 & 4 & 5 & 6 & 7 & 8 & 9 & 10 & 11 & 12 & 13 & 14 & 15 \\
\midrule
Base & 1.09 & 1.09 & 1.07 & 1.05 & 1.08 & 1.04 & 1.05 & 1.05 & 1.04 & 1.04 & 1.05 & 1.05 & 1.04 \\ 
Largewarehouse & 1.10 & 1.09 & 1.09 & 1.08 & 1.06 & 1.05 & 1.06 & 1.05 & 1.04 & 1.04 & 1.06 & 1.05 & 1.06 \\ 
Lesscrossaisles & 1.11 & 1.11 & 1.11 & 1.07 & 1.09 & 1.07 & 1.06 & 1.06 & 1.07 & 1.05 & 1.07 & 1.07 & 1.06 \\ 
Classbaseddispersion & 1.08 & 1.10 & 1.05 & 1.05 & 1.08 & 1.05 & 1.03 & 1.04 & 1.03 & 1.03 & 1.04 & 1.05 & 1.05 \\ 
Smallorders & 1.11 & 1.10 & 1.08 & 1.06 & 1.06 & 1.04 & 1.05 & 1.04 & 1.04 & 1.03 & 1.03 & 1.04 & 1.03 \\ 
Largeorders & 1.07 & 1.06 & 1.08 & 1.05 & 1.06 & *1.07 & *1.05 & -- & -- & -- & -- & -- & -- \\ 
Smallbatches & 1.10 & 1.08 & 1.07 & 1.05 & 1.07 & 1.05 & 1.04 & 1.04 & 1.05 & 1.04 & 1.05 & 1.04 & 1.03 \\ 
Largebatches  & 1.09 & 1.08 & 1.08 & 1.06 & 1.07 & 1.04 & 1.05 & *1.05 & -- & -- & -- & -- & -- \\ 
Smallrate & 1.07 & 1.08 & 1.07 & 1.04 & 1.04 & 1.04 & 1.04 & 1.04 & 1.04 & 1.04 & 1.03 & 1.04 &1.04 \\ 
Largerate  & 1.12 & 1.08 & 1.09 & 1.04 & 1.08 & 1.06 & 1.06 & 1.05 & 1.06 & 1.05 & 1.06 & 1.06 & *1.04 \\ 
\midrule
Total  & 1.09 & 1.09 & 1.08 & 1.06  & 1.07 & 1.05 &1.05 & 1.05 & 1.05 &  1.04 & 1.05 & 1.05 & 1.04 \\
\midrule[\heavyrulewidth]
\multicolumn{14}{l}{\scriptsize{\textit{Note.} $^*$ Results for less than 10 (out of 10) instances. CIOPT did not solve 1-2 instances to optimality.}}\\
\multicolumn{14}{l}{\scriptsize{\quad\quad~~ -- No results reported, since CIOPT solved less than 8 instances to optimality.}}
  \end{tabular}
}

\end{table}

\end{document}